\documentclass{amsart}[35pt]
\usepackage{amsmath, amsthm, amscd, amssymb, amsfonts, inputenc, tikz}
\UseRawInputEncoding
\DeclareMathOperator{\St}{St}
\DeclareMathOperator{\Pk}{Pk}
\DeclareMathOperator{\Pf}{Pf}
\DeclareMathOperator{\Upk}{Upk}

\DeclareMathOperator{\Coc}{Coc}
\usetikzlibrary{shapes,arrows}

\usepackage[margin=3cm]{geometry}

\newtheorem{theorem}{Theorem}[section]
\newtheorem{lem}[theorem]{Lemma}
\newtheorem{prop}[theorem]{Proposition}
\newtheorem{cor}[theorem]{Corollary}
\theoremstyle{definition}
\newtheorem{defn}[theorem]{Definition}
\newtheorem{ex}[theorem]{Example}

\theoremstyle{remark}
\newtheorem{rem}[theorem]{Remark}
\numberwithin{equation}{section}
\begin{document}

\title[A generalized theory of ideals
]{A generalized theory of ideals and algebraic substructures}

\author[M.H. Hooshmand]{M.H. Hooshmand }

\address{Department of Mathematics, Shiraz Branch, Islamic Azad University, Shiraz, Iran}

\email{\tt hadi.hooshmand@gmail.com , MH.Hooshmand@iau.ac.ir}


\date{}
\begin{abstract}
In 2011, a topic containing the concepts of upper and lower periodic subsets of (basic) algebraic structures was introduced and studied. The concept of ``upper periodic subsets'' can be considered as a generalized topic of ideals and sub-structures (e.g., subgroups, sub-semigroups, sub-magmas, sub-rings, etc.). Hence, it can be improved to a theory in every algebraic structure which extends many basic concepts including the ideals. This paper follows the mentioned goals and studies related to algebraic and topological aspects. For this purpose, first, we state an improved fundamental theorem about the unique direct representation of left upper periodic subsets and then introduce some extensions and generalizations.
As a result of the study, we classify sub-semigroups and subgroups of real numbers and introduce some related topics and concenterable upper periodic subsets.
We end the study with many research projects and some future directions of the theory.
\end{abstract}
\maketitle
\section{Introduction and Preliminaries}
\noindent
Upper, lower, and ordinary periodic subsets of basic algebraic structures (e.g., magmas,
semigroups, groups, etc.) were introduced and studied in 2011 \cite{ulpss}.
But its appearance (in the additive group of real numbers) is from the topic of ``Limit summability of real functions''
which the author raised in 2000.
In that topic, we came across subsets of real numbers $A$ with the property $A=A+1$, at first,
and then $A+1\subseteq A$, $A\subseteq A+1$. We proved that not only the general form but also the unique direct
representation of such real subsets are
\begin{equation}
A=\mathbb{Z}\dot{+}D\; , \; A=(\mathbb{Z}\dot{+}D)\dot{\cup}(\mathbb{Z}_+^0\dot{+}E)\; , \;
A=(\mathbb{Z}\dot{+}D)\dot{\cup}(\mathbb{Z}_-^0\dot{+}E),
\end{equation}
respectively, where
$D\subseteq [0,1)$, $E\subseteq \mathbb{R}$,
 $\dot{\cup}$ is the disjoint union, and $\dot{+}$ denotes the direct summation.
 Moreover, $E$ is anti integer-transference (i.e., $E\cap (E+\mathbb{Z}^*)=\emptyset$),
this is a conclusion of the directness of the summation $\mathbb{Z}_-^0\dot{+}E$ (although it is not mentioned in \cite{ulpss}).
 Thereafter we considered a generalization of such subsets via replacing the element 1 (the singleton $\{1\}$)
 by a fixed subset $B$ that are $B+A=A$, $B+A\subseteq A$ and $A\subseteq B+A$. This generalization should be more
 useful and notable if we let $A$ and $B$ be arbitrary subsets of an algebraic structure.
 Hence we arrive at the following definition in their most natural general framework.
 \begin{defn}[\cite{ulpss}]
 Suppose that $A$ and $B$ are subsets of a magma $X$ (i.e., an
arbitrary set with a binary operation ``$\cdot$''). We call $A$ {\em left
 upper  $B$-periodic} (resp. {\em left lower  $B$-periodic}) if $BA\subseteq A$ (resp. $A\subseteq BA$).
 The set $A$ is called {\em left $B$-periodic} if it is both left
upper and lower $B$-periodic (equivalently $BA=A$). The right and two-sided cases are defined analogously.
  \end{defn}
The above definition generalizes the concepts of ideals and sub-structures (sub-magmas, sub-semigroups, subgroups,  etc),
  by letting $B=X$ and $B=A$, respectively. For studying such a basic topic we need many essential
  concepts and tools such as periodic kernel (core), summand set, the start of a subset, etc,
   will be considered in continuation.
They also play the main roles for general form, unique and direct
  representations of periodic type subsets in all algebraic structures.
It is worth noting that they are the results of years of efforts and research some of which can be found in \cite{ulpss, grplike, glag, new} and
their references.\\
Since in \cite{ulpss} there are many new notations and definitions
which makes it somewhat difficult to read (and also there are some other problems),
in this paper, we try to:\\
1. reduce the notations and definitions, and express them as simply and fluently as possible;\\
2. improve and expand the topic, and try to eliminate the mistakes in the topic paper;\\
3. present some important new results, and study the related topology and topological algebraic structures; \\
4. introduce a list of upcoming projects, questions, and open problems of the theory.\\
It is worth noting that this topic can also develop the topic of power
semigroups (\cite{power}) and related equations and inequalities (e.g., $BY=A$ and $BY\subseteq A$ where $Y$ is unknown,
see Project IV in this paper)
 since the subsets of $X$ can be considered as members of the power magma $\mathcal{P}(X)$ whose
  multiplication of two members is the same as their multiplication as subsets of $X$.
\subsection{Notations and conventions} In this paper, we consider $S$ as a semigroup,
$M$ a monoid (with the identity $1=1_M$), $G$ a group, and (more generally) $X$ represents a magma
(with an arbitrary binary operation ``$\cdot$'').
The additive notation of the binary operation is $+$ with the assumption that it is commutative.
If $X$ contains an identity $1=1_X$ (e.g., if $X=M$ is a monoid), then it is called unitary.
We consider $A,B,C,D,E,F,Y$ as subsets of $X$, $\mathcal{P}(X)$ or $2^X$ the power set of $X$,
and put  $A\setminus B:=\{a\in A|a\notin B\}$,
$A^c:=X\setminus A$, and $A\dot{\cup} B$  denotes the disjoint union of $A$ and $B$.\\
We call $X$ left $B$-cancellative if $bx_1=bx_2$ implies $x_1=x_2$ for all $b\in B$ and $x_1,x_2\in X$.\\
When we use the notation $H\dot{\leq} X$ it means
$H$ is a nonempty subset of $X$ that is closed under the binary operation.
Hence, for example, $H$ is a sub-semigroup of $S$ if
and only if $H\dot{\leq} S$. It is interesting to know that $H\dot{\leq} X$ if and only if $H$ is left (right) upper
$B$-periodic for every $B\subseteq H$.\\
By $H\leq X$ we mean $H\dot{\leq} X$ and $H$ is a group (together with the restricted binary operation).
Note that if $X$ is unitary and left (or right) $H$-cancellative, then  $H\leq X$ implies $1_H=1_X$.\\
In this paper, we first fix a left identity $l$ and then
 put $B^1:=B\cup \{l\}$ and $B^*:=B\setminus \{l\}$.
 Hence, in this paper, by
$B^1$ we mean $B\cup \{l\}$ if there is a (fixed) left identity in $X$, and otherwise
$B^1:=B\cup \{1\}\subseteq X^1$ (in the additive notation $B^1$ is replaced by $B^0$).
Note that $B^1A=BA\cup A$, and so $B^1B=B\neq \emptyset$ if and only if $B\dot{\leq} X$.\\
By $H\dot{\leq}_\ell X$ we mean
$H\dot{\leq} X$ and $H$ contains a (fixed) left identity of $X$.
Hence $H\leq_\ell X$ means $H\leq X$ and $1_H$ is a left identity of $X$, and call it a left subgroup of $H$.
Note that in monoids (resp. groups), $H\dot{\leq}_\ell M$ (resp. $H\leq_\ell G$) if and only if  $H$ is a sub-monoid
(resp. subgroup) of $M$ (resp. $G$).\\
It is worth noting that if $H\leq_\ell S$, then $S$ is left $H$-cancellative
(we use this fact in the paper, repeatedly).
Indeed, if $H\leq_\ell S$ and $h_1x_1=h_2x_2$, for $x_1,x_2\in S$ and $h_1,h_2\in B$, then
\begin{equation}
x_1=1_Hx_1=h_1^{-1}(h_1x_1)=h_1^{-1}(h_2x_2)=(h_1^{-1}h_2)x_2
\end{equation}
where $h_i^{-1}$ ($i=1,2$) is the (unique) inverse of $h_i$ in the group $H$.\\
If $B\subseteq S$, then by $\langle B\rangle$ (resp. $\langle \langle B\rangle\rangle$) we mean the sub-semigroup
(resp. sub-group, if exists)
of $S$ generated by $B$. During the paper, using the notation $\langle \langle B\rangle\rangle$ carries the assumption
of the existence of a subgroup containing $B$ (and so the existence of $\langle \langle B\rangle\rangle$).
If $l\in B$, for some fixed left identity $l$ of $S$,
 then $\langle B\rangle^1=\langle B\rangle$ and $B\langle B\rangle=\langle B\rangle B=\langle B\rangle$.
\subsection{Inverses of subsets}
Introducing the concept of inverses of subsets of magmas allows us to develop many important results
from groups to a wider class of semigroups.
When we use the notation $B_\ell^{-1}(l)$ or simply
$B_\ell^{-1}$ or $B^{-1}$ (if $l$ is a well known left identity of $X$ and there is no any risk of
confusion), we suppose that\\
(a) for every $b\in B$ there exists $\beta\in X$ such that $\beta b=l$ (i.e., every $b\in B$ is left invertible relative to $l$),\\
(b) $B^{-1}$ is one of the \underline{minimal} subsets
$Y\subseteq X$ with this property:
$$
(\forall b\in B)\;  (\exists y\in Y) \; : \; yb =l.
$$
Such elements $y$ are denoted by $b_\ell^{-1}(l)$ or simply $b_\ell^{-1}$, $b^{-1}$.
Note that here
we do not need the two-sided identity $1$, and the existence of a (fixed) left
identity $l$ is enough for defining $B^{-1}$ and $b^{-1}$. \\
Now, let  $B^{-1}$ exists,
then $B^{-1}$ satisfies the following properties
$$
(\forall b\in B)\;  (\exists  \beta\in B^{-1}) \; : \; \beta b =l\; , \; (\forall \beta\in B^{-1})\;  (\exists b\in B) \; : \; \beta b =l.
$$
If $X$ is right $B$-cancelative and $B^{-1}$ exists, then
\begin{equation}
B^{-1}=\{\beta\in X: \beta b=l\; \mbox{for some } b \in B \},
\end{equation}
which implies $B^{-1}$ is unique. \\
We use this fact repeatedly that if $B\subseteq H\leq_\ell S$, then $B^{-1}=B^{-1}(1_H)$ exists but it may not unique in $S$
(although it is unique in $H$, see the next example). Of course if $1_H$ is  a right identity for $S$, or $S$ is right $B$-cancelative, then it is also unique
in $S$ and it is as the same as $B^{-1}$ in the group $H$  .\\
During the paper, using the notation $B^{-1}$ carries the assumption
of existence of an inverse of $B$ relative to a fixed left identity $l$.
If it is the case, then $B$ is called $l$-symmetric (or symmetric) if $B=B^{-1}$,
and $B$ is called anti $l$-symmetric (or anti $l$-inversion) if $B\cap B^{-1}=\emptyset$.\\
\begin{ex}
Consider the semigroup $(S,\bar{\cdot})$ where $S=\mathbb{R}\setminus\{0\}$ and
$x\bar{\cdot} y=|x|y$. It contains the left identities $\pm 1$, and no right identity. It is
left but not right cancelative.
Every element has a unique  right inverse relative to each left identity (i.e., $\frac{\pm 1}{|x|})$,
but there is not  a left inverse relative to them for some elements (since it is not a group).
By putting $H=(0,+\infty)$ and fixing $l=1$, it is obvious that $H\leq_\ell S$, $l=1_H$. It is interesting to know that $H^c=(-\infty,0)$
is another left subgroup (where $l=-1$). Hence for every $B\subseteq H$, $B^{-1}=B^{-1}_\ell(1)$ exists,
and it can be one of the sets $\{\frac{1}{b}:b\in B\}$ or $\{\frac{-1}{b}:b\in B\}$, or
all subsets of $S$ obtained by replacing some elements of these two sets by their (usual) additive inverses
(and all are minimal relative to the property).
\end{ex}
\subsection{Direct product of subsets and factors of algebraic structures}
We say the product $AB$ is direct and denote it by $A \cdot B$
if the representation of each $x\in AB$ as $x=ab$, $a\in A$, $b\in B$ is unique. Hence,
$Y=A \cdot B$ if and only if $Y=A B$ and the product $A B$ is direct
(a factorization of $Y$ by two subsets of $X$, the additive notation is $Y=A\dot{+} B$).
If $X=A \cdot B$, then $A$ (resp. $B$) is called a left (resp. right) factor of $X$ relative to
$B$ (resp. $A$). For example, every subgroup of a group is a left (resp. right) factor
relative to its right (resp. left) transversal, and so it is a two-sided factor.\\
By $H\leq_{\ell f} X$ we mean $H\leq X$ and $H$ is a left factor of $X$, and call $H$ a
left factor subgroup of $X$. In a vast class of magmas including all semigroups, every
left factor subgroup is a left subgroup, but the converse is not true even in semigroups (see (vi)).
 Note that in groups, $H\leq_{\ell f} G$ if and only if $H\leq_{\ell} G$,
and if and only if $H\leq G$ .\\
We have the following basic properties (e.g., see \cite{ulpss, glag, new}):\\
{\bf (i)} If $A$ and $B$ are finite subsets
of $X$, then $AB=A\cdot B$
if and only if $|AB|=|A||B|$.\\
We have introduced a related conjecture in 2014 that says:
for every finite group $G$ and every factorization $|G|=ab$
there exist subsets $A,B\subseteq G$ with $|A|=a$ and $|B|=b$ such that $G=AB$.
Until now, it is proved for all finite solvable groups, finite groups such that for every divisor $d$ of their order, there is
a subgroup of order or index $d$, and all groups of order $\leq 10000$ (but it is still open).\\
{\bf (ii)} If  $\emptyset\neq A,B\subseteq M$
 (and $A^{-1}$, $B ^{-1}$ exist), then
$$AB =A\cdot B\Leftrightarrow A^{-1}A\cap BB^{-1}=\{ 1\}.$$
{\bf (iii)}
\begin{lem}
If $H\leq_\ell S$, then
$$HA=H\cdot A\Leftrightarrow A\cap H^{*}A=\emptyset$$
\end{lem}
(Note that $A\cap BA=\emptyset$ means ``$A$ is anti-left $B$-transference''
which has a meaningful and important interpretation and explanation, see \cite[p. 454, 458]{ulpss}. Therefore,
the product $HA$ is direct if and only if $A$ is anti-left  $H^*$-transference.)
\begin{proof}
If $A\cap (H^{*}A)=\emptyset$, $h_1a_1=h_2a_2$ where $h_1,h_2\in H$, $a_1,a_2\in A$, and $h_1\neq h_2$,
then $a_1\in A\cap (H^{*}A)$ that is a contradiction. Conversely, if $HA=H\cdot A$ and $a_0\in A\cap (H^{*}A)$,
then $1_Ha_0=a_0=h'a'$ for some $h'\in H^*$ and $a'\in A$, which requires $h'=1_H$, a contradiction.\\
\end{proof}
{\bf (iv)}  If $N\unlhd G$, then $G=N\cdot Y=Y\cdot N$,
for every transversal $Y$ of $N$.\\
{\bf (v)} A sub-semigroup of a group $G$ is a left or right factor if and if only if it is a subgroup.
So, if $H\dot{\leq} G$ and
$H\nleq G$, then $H$ is not  a left or a right factor of $G$.\\
{\bf (vi)} A subgroup $H$ of a semigroup $S$
is a left factor (i.e., $H\leq_{\ell f}S$) if and only if the equation $s=xs$
has the only solution $x=1_{H}$ in $H$, for
every $s\in S$ (see \cite[Theorem 2.2]{glag}). Hence,
$H\leq_{\ell f}S$ implies $H\leq_\ell S$ but the converse is not true, e.g.,
consider $S=F_C$ (of order 27) and $H=U$ in the last part of Example 1.5.
\begin{rem}
Due to (vi), for a given left factor subgroup $\mathcal{B}$ of $S$
(e.g., if $S=M$ is a right cancellative monoid and $\mathcal{B}$
an arbitrary subgroup of $M$)
we fix one of subsets $Y$ such that $S=\mathcal{B}\cdot Y$  and denote
it by $\mathcal{D}$, hence $S=\mathcal{B}\cdot \mathcal{D}$ is a fixed factorization for $S$.
We also denote it by $\mathcal{D}_b$ if $\mathcal{B}=\langle\langle b\rangle\rangle$.
For example, if $S$ is the additive group of real numbers
and $\mathcal{B}=\mathbb{Z}$, then there are uncountable
choices for $Y$ including all intervals $[n,n+1)$ where $n\in \mathbb{Z}$.
In this case we can consider $\mathcal{D}=\mathcal{D}_1=[0,1)$ and fix it, so we have the determined factorization
$\mathbb{R}=\mathbb{Z}\dot{+} [0,1) $ during the study.\\
As another example, let $S=\mathbb{R}\setminus\{0\}$ be the semigroup mentioned in Example 1.2. Then, $\mathcal{B}=(0,+\infty)$
can be considered as the fixed left factor subgroup, $\mathcal{D}=\{-1,1\}$, and we have
$S=(0,+\infty)\dot{\circ} \{-1,1\}$.
\\
In the sequel we consider $\mathcal{B}$ (the fixed left factor subgroup of $S$, if exists) and $\mathcal{D}$
(the right fixed transversal of $\mathcal{B}$) as mentioned above, and
use it repeatedly in this paper.
Note that since $1_{\mathcal{B}}$ is a left identity of $S$ and $\mathcal{B} \mathcal{D}=\mathcal{B}\cdot \mathcal{D}$, Lemma 1.3
implies $\mathcal{D}\cap \mathcal{B}^*\mathcal{D}=\emptyset$.
\end{rem}
Since the fixed subgroup $\mathcal{B}$ with the property plays an important role in the topic, we here
give a vast class of such subgroups in the next example.
\begin{ex}
Let $M$ be a monoid and $R:=R(M)$ (resp. $L:=L(M)$) the sub-monoid of all its right (resp. left)
invertible elements. It is clear that $U:=U(M)$ (the set of all invertible elements of $M$) is
the same $L\cap R$ . By (vi),  $U$ is a right (resp. left)
factor subgroup of $R$ (resp. $L$). Note that, in general, $U$ is not a right or left
factor subgroup of $M$. For instance, let $C$ be a non-empty set and $M=F_C$ the monoid of all
functions from $C$ into $C$. Then $R$ (resp. $L$) is the sub-monoid of all
surjective (resp. injective) functions of $M$, and so $U$ is
the group of all bijections.
Therefore, $U$ is a right (resp. left) factor subgroup of the monoid  $R$ (resp. $L$).
But $U$ (and also $R$, $L$) is not necessarily a left or right factor of $M=F_C$.
For if $C=\{ 1,2,3\}$, then $|M|=|F_C|=27$, $|U|=|R|=|L|=6$,
and $6$ does not divide $27$.
\end{ex}
\section{Periodic and upper, lower periodic subsets}
In Definition 1.1, we recall the conception of periodic type subsets of algebraic structures from \cite{ulpss}.
Now, we present some other primiparities that we need for improving the unique direct representation
of left upper periodic subsets (as a fundamental theorem of the theory). \\
{\bf (1)} $H\dot{\leq} X$ if and only if $H\neq\emptyset$ is upper $B$-periodic for every $B\subseteq H$. Moreover, for $H\neq\emptyset$ we have
$$
H\dot{\leq} X \Leftrightarrow H^1H=H \Leftrightarrow HH^1=H \Leftrightarrow \mbox{ $H$ is left or right upper $H$-periodic}
$$
{\bf (2)} $I\neq\emptyset$ is a left (resp. right) ideal of $X$ if and only if $I$ is left (resp. right) upper $X$-periodic.\\
{\bf (3)} The union and intersection of a family of left upper $B$-periodic subsets of $X$ is so (see Subsection 2.2 for
more information).\\
{\bf (4)} $BA\subseteq A$ (resp. $A\subseteq BA$) if and only if $A=B^1A$ (resp. $BA=B^1A$)
as subsets of $X^1$.\\
{\bf (5)} If $B\dot{\leq} S$, then all sets of the form $A=B^1E$ and $A=BE$ ($E\subseteq S$)
are left upper $B$-periodic. Moreover, if $E\nsubseteq BE$ (e.g., if $A$ is anti-left $B$-transference), then
$A=B^1E$ satisfies $BA\subset A$ which means $A$ is strictly left upper $B$-periodic.  \\
{\bf (6)} If $A\subseteq S$ is left upper $B$-periodic, then $AY$ is so, for every $Y\subseteq S$. \\
{\bf (7)} $A\subseteq S$ is left upper $B$-periodic if and only if it is left upper $\langle B\rangle$-periodic
(see Proposition 3.2).\\
For instance, in $(\mathbb{R},+)$, by using the mentioned properties, we have
$$
1+A\subseteq A\Leftrightarrow \{0,1\}+A=A\Leftrightarrow \mathbb{Z}_+^0+A=A \Leftrightarrow \mathbb{Z}_++A\subseteq A,
$$
since $B=\{1\}$, $B^0=\{0,1\}$, $\langle B\rangle=\mathbb{Z}_+$, and $\langle B\rangle^0=\mathbb{Z}_+^0$.\\
It is also interesting to know that
$$
1+A= A\Leftrightarrow \mathbb{Z}_++A=A \Leftrightarrow \mathbb{Z}+A= A\Leftrightarrow -1+A=A
\Leftrightarrow \mathbb{Z}_-+A=A
$$
Note that if $\mathbb{Z}_++A=A$ and $a\in A$, then $a=n_0+a_0$ for some integer $n_0>0$ and $a_0\in A$.
Thus
$$
a-1\in (a-n_0)+\mathbb{Z}^0_+=a_0+\mathbb{Z}^0_+\subseteq A+\mathbb{Z}^0_+=A,
$$
and so $A\subseteq A+1$ (also, $A+1\subseteq A$ obviously).\\
{\bf (8)}
If $H\leq G$, then $H^c$ is $H$-periodic.
This is not true for a semigroup $S$, but if $H\leq_\ell S$, then $H^c$ is left $H$-periodic
(analogously for the right case). Therefore, in the semigroup $S=\mathbb{R}\setminus\{0\}$ in Example 1.2,
if $H=(0,+\infty)$, then $H^c$ is left $H$-periodic but not  right $H$-periodic.\\
{\bf (9)} If $B^{-1}\subseteq B$, for some $B^{-1}$ (e.g., if $B$ is symmetric), then  $A$ is left upper  $B$-periodic if and only if $A^c$ is so
. For if $BA\subseteq A$ and $b\alpha\in A$ for some $b\in B$ and $\alpha\in A^c$, then
$$
\alpha=l \alpha=b^{-1}(b\alpha)\in B^{-1}A\subseteq BA\subseteq A,
$$
that is a contradiction. Hence $BA^c\subseteq A^c$, and the converse is similar to prove.\\
{\bf (10)} If $l\in B^{-1}\subseteq B$, for some $B^{-1}$ (e.g., if $B\leq_\ell S$), then  $A$ is left  $B$-periodic if and only if $A^c$ is so
(by (9) and due to $A=lA\subseteq BA$).\\
The following example gives some rational-periodic type sets in the additive semigroup of real numbers.
 \begin{ex}
In $(\mathbb{R},+)$, every open (resp. closed) interval of the form
$(M,+\infty)$ (resp. $[M,+\infty)$) is $\mathbb{Q}_+$-periodic (resp. upper $\mathbb{Q}_+$-periodic).
Both are upper $r$-periodic but not $r$-periodic, for every $r\in \mathbb{Q}_+$. They are also
 lower $\mathbb{Q}_-$-periodic (but not $\mathbb{Q}_-$-periodic).
 \end{ex}
In continuation, several important theorems and results are stated in
which $\mathcal{B}$ and $\mathcal{D}$
 are the ones that have been described in Remark 1.4.
\begin{lem}[Unique direct representation of the left $\mathcal{B}$-periodic subsets]
Every left $\mathcal{B}$-periodic subset
 $A$ of $S$ has the unique direct representation $($factorization$)$
$A=\mathcal{B}\cdot D$ with $D\subseteq \mathcal{D}$ $($i.e., $D$ is a $\mathcal{D}$-set$)$.
Hence, $\{\mathcal{B}D:D\subseteq \mathcal{D}\}$ determines  all left $\mathcal{B}$-periodic subsets, exactly.
\end{lem}
\begin{proof}
If $\mathcal{B}A=A$, then putting $D:=\mathcal{B}\cap A$ we have
$A=\mathcal{B}\cdot D$  and $D\subseteq \mathcal{D}$. For if $a\in A$, then $a=\beta d$, for some $\beta\in \mathcal{B}$ and $d\in \mathcal{D}$
(by Remark 1.4),
and so $d=\beta^{-1}a\in \mathcal{B}A=A$. Thus  $A\subseteq \mathcal{B}D$, and also $\mathcal{B}D\subseteq \mathcal{B}A=A$.
This means $A=\mathcal{B}D\subseteq \mathcal{B}\cdot \mathcal{D}$ which implies $A=\mathcal{B}\cdot D$.
For uniqueness, let $A=\mathcal{B}\cdot D_1=\mathcal{B}\cdot
D_2$, where $D_1, D_2$ are $\mathcal{D}$-sets. If $d_1\in D_1$, then $1_{\mathcal{B}}d_1=d_1=bd_2\in\mathcal{B}\cdot \mathcal{D}$,
for some
$b\in \mathcal{B}$ and $d_2\in D_2$, and so $d_1=d_2\in D_2$ which means $D_1\subseteq D_2$ (and similarly $D_2\subseteq D_1$).
\end{proof}
\begin{cor}
In finite groups $G$, the upper periodicity of subsets and periodicity are the same. Moreover, for every $B\neq\emptyset$,
$A$ is left upper $B$-periodic if and only if it is  left $\langle \langle B\rangle\rangle$-periodic,
and we have the unique direct representation mentioned in Lemma 2.2.
$($similar properties hold for subsets $B$ of semigroups such that $\langle B\rangle=\langle \langle B\rangle\rangle).$
\end{cor}
\begin{proof}
If $BA\subseteq A$, then Proposition 3.2(a) and Lemma 2.2 together with $\langle B\rangle=\langle \langle B\rangle\rangle$
(since $G$ is finite) imply $A=BA=\langle B\rangle A$, and the properties hold.
\end{proof}
\begin{ex}
If $A\subseteq (\mathbb{R},+)$ is $\mathbb{Z}$-periodic (equivalently $A+1=A$), then it has the unique
direct representation (and general form) $A=\mathbb{Z}\dot{+}D$ where $D\subseteq [0,1)$.
The only upper 1-periodic subsets of $(\mathbb{Z}_n,+)$ are itself and the empty set. For if
$1+A\subseteq A$, then Corollary 2.3 requires that $A=\mathbb{Z}_n+A$.\\
Let $S$ be as mentioned in Example 1.2. Then,
the only left $(0,+\infty)$-periodic subsets of $S$ are $\emptyset$, $(0,+\infty)$,
$(-\infty,0)$, and $\mathbb{R}\setminus\{0\}$ (because $D\subseteq \mathcal{D}=\{-1,1\}$ in Lemma 2.2).
\end{ex}
In \cite[Theorem 2.6]{ulpss}  the unique direct representation of the left upper $B$-periodic
subsets of semigroups (and groups) were introduced and proved as a main result.
In the next main theorem, we improve it and prove the theorem in more understandable form. Also,\\
1. We merge the specific case ($B^1=\mathcal{B}$) into the general case,
and show that $B$ must be a sub-semigroup of $S$.\\
2. We prove that its main condition (that is equivalent to  $(2.4)$) is not only sufficient but also necessary.\\
3. A problem in its proof (four lines left to the end of the proof, i.e. the equality $E_1\cap BE_2=E_1\cap BE_1$,
which has no reason to be correct) will be removed.\\
4. The initial assumption $(\mathcal{B}\setminus B)^{-1}\setminus \{1_{\mathcal{B}}\}\subseteq B$ will be characterized completely, by introducing:\\
{\bf Positive-negative partition of groups.}
\begin{defn}
Let $G$ be a group and denote by $G_2$ the set of all $x$ such that $x^2=1$.
By a positive-negative partition of $G$, we mean a disjoint union $G=G_+\dot{\cup} G_- \dot{\cup} G_2$
such that $G_-=(G_+)^{-1}$. We call $B\subseteq G$ a positive subset if and only if $B=G_+$ for some
positive-negative partition of $G$. By a positive sub-semigroup, we mean a positive subset that is also a sub-semigroup.
\end{defn}
It is obvious that $B$ is a positive subset if and only if $B^{-1}$ is so, and $G=G_2$ if and only if some
(all) positive subsets of $G$ are empty. Also, $G_+\dot{\leq} G$ if and only if $G_-\dot{\leq} G$.\\
The positive-negative partitions of groups always exist. Define the equivalent relation
$\sim$ by $x\sim y$ if and only if $x\in\{y,y^{-1}\}$, for all $x,y\in G$. Then $G\setminus G_2$ is either empty or has a partition
with two members classes induced by $\sim$. Now, define $G_+$ as an arbitrary element of the partition of $G\setminus G_2$ if
$G\neq G_2$, otherwise $G_+=\emptyset$. Then we have $G_-=(G\setminus G_2)\setminus G_+$. Therefore, $G=G_+\dot{\cup} G_- \dot{\cup} G_2$
is a positive-negative partition for $G$. \\
There is an interesting connection between this concept and the orderable groups. Indeed,
if $G$ is orderable (e.g., if $G$ is a torsion-free abelian group), then there is a positive sub-semigroup $G_+$ of $G$
(that is a positive cone, see \cite{ordered}) and vice versa.\\
\begin{ex}
Let $G$ be a subgroup of $(\mathbb{R},+)$. Then the usual positive subset of $G$ is $G_+=G\cap (0,+\infty)$
(with the complement $G_-=G\cap (-\infty,0)$)
which is also a positive sub-semigroup.
It is interesting to know that if $G\leq (\mathbb{Q},+)$, then such $G_\pm$ are
the only positive sub-semigroups of $G$ (e.g., $\mathbb{Z}_\pm$, $\mathbb{Q}_\pm$, etc.).
For if $1\in G_+$ (without loss of the generality), then $\mathbb{Z}_+\subseteq G_+$. Now, if
$-\frac{m}{n}\in G_+$, for some $m,n\in \mathbb{N}$, then $-m\in G_+$ that is a contradiction.\\
{\bf Question 0.} What about positive sub-semigroups of real (complex) numbers (is there a similar property for them)?\\
For complex numbers, we can define
$$\mathbb{C}_+:=\{a+bi\in \mathbb{C}: a\geq 0 \}\setminus \{bi: b\leq 0 \}=\{a+bi\in \mathbb{C}: a> 0 \}\cup \{bi: b> 0 \}$$
and $\mathbb{C}_-:=-\mathbb{C}_+=\{a+bi\in \mathbb{C}: a\leq 0 \}\setminus \{bi: b\geq 0 \}$. It
is obvious that $\mathbb{C}=\mathbb{C}_+\dot{\cup}\mathbb{C}_-\dot{\cup} \{0\}$ and  $\mathbb{C}_+$
is a positive sub-semigroup of complex numbers.
\end{ex}
Now, what is the necessary and sufficient condition for $B$ to be a positive subset of $G$?\\
For the necessary condition, we have
$$
B \mbox{ is a positive subset of } G\Rightarrow B \mbox{ is anti symmetric }\Rightarrow  B\cap G_2=\emptyset
$$
but the converses of the above implications are not valid.
Also, for some necessary and sufficient conditions of positive subsets, we state and prove the next proposition.
\begin{prop} Let $B$ be a subset of $G$. Then\\
$($a$)$ $B$ is a positive subset if and only if $(G\setminus B)^{-1}\setminus G_2= B$ $($equivalently $(G_2^c\setminus B)^{-1}=B$$)$
 and $B\cap G_2=\emptyset$.\\
$($b$)$ If $B$ is antisymmetric, then
\begin{equation}
\mbox{ $B$ is a positive subset of $G$} \Leftrightarrow  (G\setminus B)^{-1}\setminus G_2\subseteq B \Leftrightarrow
(G\setminus B)^{-1}\setminus G_2=B\Leftrightarrow (G_2^c\setminus B)^{-1}=B
\end{equation}
Hence, the following statements are equivalent $($if $B\cap B^{-1}=\emptyset$$)$:\\
 $($i$)$ $(G\setminus B)^{-1}\setminus \{1\}\subseteq B$;\\
$($ii$)$ $G_2=\{1\}$ and $B$ is a positive subset of $G$;\\
 $($iii$)$ $G^*\setminus B\subseteq B^{-1}$;\\
$($iv$)$ $(G\setminus B)^{-1}\setminus \{1\}= B$;\\
$($v$)$ $G^*\setminus B=B^{-1}$.
 \end{prop}
 \begin{proof}
Let $B$ be a positive subset of $G$. Then $G=B\dot{\cup} B^{-1}\dot{\cup} G_2$ and $B\cap G_2=\emptyset$. Now, we have
$$
G\setminus B=B^{-1}\dot{\cup} G_2\Rightarrow (G\setminus B)^{-1}=B\dot{\cup} G_2
\Rightarrow (G\setminus B)^{-1}\setminus G_2= B
$$
Conversely,
$(G\setminus B)^{-1}\setminus G_2= B$ together with $B\cap G_2=\emptyset$ imply $G_2\subseteq (G\setminus B)^{-1}$ and
$$
(G\setminus B)^{-1}=(G\setminus B)^{-1}\cup G_2=\big((G\setminus B)^{-1}\setminus G_2\big)\cup G_2=
B\dot{\cup} G_2
$$
Hence
$$
G=B\dot{\cup}(G\setminus B)=B\dot{\cup}B^{-1}\dot{\cup} G_2,
$$
and so $B$ is a positive subset.\\
Note that we have the identities
$$
(G\setminus B)^{-1}\setminus G_2=(G\setminus B^{-1})\setminus G_2=(G\setminus G_2)\setminus B^{-1}=G_2^c\setminus B^{-1}
=(G_2^c\setminus B)^{-1}
$$
Thus, for proving (b) it is enough to show that $G_2^c\setminus B^{-1}\subseteq B$ requires $B\subseteq G_2^c\setminus B^{-1}$.
To show it, first note that $B\cap B^{-1}=\emptyset$ imply $B\cap G_2=\emptyset$ and so $B\subseteq G_2^c$,
$B^{-1}\subseteq G_2^c=B\dot{\cup} (G_2^c\setminus B)$. Thus $B^{-1}\subseteq G_2^c\setminus B$, since $B\cap B^{-1}=\emptyset$,
and hence $B\subseteq (G_2^c\setminus B)^{-1}=G_2^c\setminus B^{-1}$.\\
To prove the last part, first the condition (i) with $B\cap B^{-1}=\emptyset$ imply $G=B\dot{\cup}B^{-1}\dot{\cup} \{1\}$ and so
$G_2=\{1\}$, $G_2^c=G^*$. Then, one can get the results by using or similar to (a) and (b).
 \end{proof}
Considering the above proposition, we can provide important equivalent conditions which is used
in the main theorem.
 \begin{lem}
If $A\subseteq S$, $H\leq_\ell S$,
and $H^*\setminus B^{-1} \subseteq B\subseteq H$ (if  $B$ is antisymmetric, then this is equivalent to: $B$ is a positive subset of $H$ and $H_2=\{1\}$), then
\begin{equation}
A\cap BA=\emptyset\Leftrightarrow A\cap H^*A=\emptyset \Leftrightarrow A\cap B^{-1}A=\emptyset
\Leftrightarrow HA =H\cdot A 
\end{equation}
\end{lem}
\begin{proof}
If $1_H\in B$, then it is obvious (because $A\cap BA=A\cap B^{-1}A=A)$).
Hence, let $B\subseteq H^*$. We have
 $$A\cap H^*A=(A\cap BA)\cup (A\cap (H^*\setminus B)A).$$
If $A\cap (H^*\setminus B)A\neq \emptyset$, then
$$
\emptyset\neq A\cap (H^*\setminus B)^{-1}A\subseteq  A\cap BA.
$$
Thus, $A\cap BA=\emptyset$ implies $A\cap H^*A=A\cap (H^*\setminus B)A= \emptyset$.
Now, one can complete the proof by Lemma 1.3.
\end{proof}
\begin{cor} Let $S$ has a $($fixed$)$ left identity $l$ and $E\subseteq S$.\\
(a)  If $l\notin B$, then
$$
B^1E=B^1\cdot E  \Leftrightarrow BE=B\cdot E\; , \; E\cap BE=\emptyset
$$
$($i.e., the product $B^1E$ is direct if and only if $BE$ is direct and $E$ is anti-left $B$-transference.$)$\\
(b) If $H\leq_\ell S$, then
$$
E\cap H^\ast E=\emptyset  \Leftrightarrow E\cap H_+ E=\emptyset
\Leftrightarrow E\cap H_- E=\emptyset \Leftrightarrow  H_+^1E=H_+^1\cdot E
\Leftrightarrow  HE=H\cdot E
$$
Also, for every $A\subseteq S$ we have
$$
H_+A\subseteq A  \Leftrightarrow A\subseteq \beta A\;\; ; \;\; \forall \beta\in H_-
\Rightarrow A\subseteq H_-A
$$
\end{cor}
\begin{rem}
According the above results we can completely characterize the condition $(\mathcal{B}\setminus B)^{-1}\setminus \{1\}\subseteq B$
in the main theorem (unique direct representation of upper periodic sets)  of \cite{ulpss} if we have $B\cap B^{-1}=\emptyset$.
Fortunately, by dividing the proof for two cases $BA=A$ (i.e., $A$ is left $B$-periodic)
and $BA\varsubsetneq A$ (i.e., $A$ is strictly left upper $B$-periodic) we can have $B\cap B^{-1}=\emptyset$ without loss
of the generality.
Because, if $BA=A$, then we do not need the condition at all, and if $BA\varsubsetneq A$, then $B$ is antisymmetric.
For if $b_0\in B\cap B^{-1}$, then $A=b_0^{-1}b_0A\subseteq b_0^{-1}A\subseteq BA$ which contradicts $BA\varsubsetneq A$.
Therefore, under the conditions of the theorem $(\mathcal{B}\setminus B)^{-1}\setminus \{1\}\subseteq B$ can be replaced by
``$B$ is a positive subset of $\mathcal{B}$ and $\mathcal{B}_2=\{1_\mathcal{B} \}$'' without any change in the generality and quality of the case
(by Proposition 2.7).
\end{rem}
Now, we are ready to introduce the improved and generalized version of the fundamental theorem of the topic.
We recall that $\mathcal{B}$ is the fixed subgroup mentioned in Remark 1.4 and so $S=\mathcal{B}\cdot \mathcal{D}$,
where $\mathcal{D}$ is a fixed right transversal of $\mathcal{B}$ in $S$.
\begin{theorem}[The unique direct representation of left upper periodic subsets]
If $B\neq\emptyset$ is a positive subset of $\mathcal{B}$ and $\mathcal{B}_2=\{1_\mathcal{B} \}$,
then every left upper $B$-periodic subset $A$ has the \underline{unique} \underline{direct} representation
\begin{equation}
A=\mathcal{B}\cdot D\dot{\cup} B^1\cdot E \;\;\; ; \;\;\; D\subseteq \mathcal{D},\; E\subseteq S
\end{equation}
if and only if
\begin{equation}
BA\cap \mathcal{B}A^c\subseteq B(A\setminus BA)
\end{equation}
$($i.e., $A$ is well left $(\mathcal{B},B)$-started, see Definition 2.17$)$.\\
Moreover, if $(2.3)$ or  $(2,4)$ holds, then
$D\dot{\cup} E$ is anti-left $\mathcal{B}^*$-transference and $B\dot{\leq} S$ $($and so $B$ is a positive sub-semigroup of $\mathcal{B}$$)$.
\end{theorem}
\begin{proof}
Let $BA\subseteq A$, $C=C_\mathcal{B}(A)$ be the largest $\mathcal{B}$-periodic subset of $A$ (such subset $C$ exists, see Subsection 2.1),
and put $D:=\mathcal{D}\cap C$, $F:=A\setminus C$ and
$E:=A\setminus BA$. Then, $D\subseteq \mathcal{D}$, $C=\mathcal{B}\cdot D$ by Lemma 2.2, and $E\cap BE=\emptyset$ (because $E\cap BA=\emptyset$).\\
If $BA=A$ (i.e., $A$ is left $B$-periodic), then $E=\emptyset$ and $(2.4)$, $(2.6)$ imply $BA\setminus C=BA\cap \mathcal{B}A^c \subseteq \emptyset$,
thus $BA\subseteq C$ and hence
$$
A\subseteq \mathcal{B}A=\mathcal{B}BA\subseteq \mathcal{B}C=C\subseteq A,
$$
thus $A=C=\mathcal{B}\cdot D$ which means $(2.3)$ is satisfied (due to $E=\emptyset$).\\
Therefore, we can assume that $BA\neq A$ thus $E\neq\emptyset$, and
of course $B\cap B^{-1}=\emptyset$ (by Remark 2.10), $B^1\neq \mathcal{B}$, $1_\mathcal{B}\notin B\neq \mathcal{B}^*$
since $\mathcal{B}^*\setminus B=B^{-1}\neq\emptyset$.\\
Then, for $F$ we have:\\
 (1) $F$ is left upper $B$-periodic. Because
 $$
 BA\subseteq A\Rightarrow B(C\dot{\cup} F)\subseteq A\Rightarrow BC\cup BF\subseteq C\dot{\cup} F,
 $$
 and hence it is enough to show that $BF\cap C=\emptyset$.
For if $by\in C$, for some $b\in B$ and $y\in F$, then $y\in b^{-1}C\subseteq \mathcal{B}C=C$ that is a contradiction. Therefore $BF\subseteq F$.\\
 (2) $F=B^1E$.  First, we have $E\subseteq F$, for if $x\in E$, then $x\notin BA$ thus
 $b^{-1}x\notin A$, for all $b\in B\subset \mathcal{B}$, which implies $x\notin C$ and thus $x\in F$.
 Hence
 $$
 B^1E\subseteq B^1F=F\cup BF=F.
 $$
 On the other hand, if $(2.4)$ holds, then
 $$F\setminus E=(A\setminus C)\cap (A\setminus BA)^c=(A\cap \mathcal{B}A^c)  \cap (A^c\cup BA)
 =BA\cap \mathcal{B}A^c\subseteq  B(A\setminus BA)\subseteq B^1E$$
Hence
$F=(F\setminus E)\cup E\subseteq B^1E\cup E=B^{1}E$.\\
(3) $F=B^1\cdot E$. To show that the product $B^1\cdot E$ is direct, let $b_1e_1=b_2e_2$ where
$b_1,b_2\in B$, $b_1\neq b_2$ and $e_1,e_2\in E$, then
$e_1=1_\mathcal{B}e_1=(b_1^{-1}b_2)e_2$. Putting $\beta=b_1^{-1}b_2$
we have $\beta\in \mathcal{B}^*$, $\beta\notin B$ (since $E\cap BE=\emptyset$), and
$$e_2=\beta^{-1}e_1\in (\mathcal{B}^*\setminus B)^{-1}E=BE\subseteq BA,$$ (by Proposition 2.7, since $B\cap B^{-1}=\emptyset$),
thus $e_1\in E\cap BA=\emptyset$, a contradiction. Therefore, $b_1=b_2$ and $e_1= 1_\mathcal{B}e_2=e_2$,
which means $BE=B\cdot E$ that implies $B^1E=B^1\cdot E$ due to $E\cap BE=\emptyset$. \\
Therefore $A=C\dot{\cup} F=\mathcal{B}\cdot D\dot{\cup} B^1\cdot E$ and we arrive at $(2.3)$.\\
For \underline{uniqueness}, note that if $A$ has the form $(2.3)$, then Theorem 3.5 together with Lemma 2.8 require
that $B\dot{\leq} S$, and so
$$
A\setminus BA=(C\cup F)\setminus (BC\cup BB^1E)=(C\cup F)\setminus (C\cup BE)=F\setminus BE=E
$$

which means $E=A\setminus BA$ is unique (with respect to $A,B$). Hence, from the equality $\mathcal{B}D=A\setminus B^1E$ and Lemma 2.2, we conclude that
$$D=(A\setminus B^1E)\cap \mathcal{D}=(A\setminus B^1(A\setminus BA))\cap \mathcal{D}$$ which means
$D$ is also unique. \\
Conversely, let $A=\mathcal{B}\cdot D\dot{\cup}B^1\cdot E$ where $D\subseteq \mathcal{D}$ and $E\subseteq S$.
Put $F:=B^1E$ and suppose that $C_\mathcal{B}(F)$ be
the largest left $\mathcal{B}$-periodic subset of $F$.
Then, we claim that $C_\mathcal{B}(F)=\emptyset$ and $C_\mathcal{B}(A)=\mathcal{B}D$. For if $E=\emptyset$, then it is obvious,
thus assume $E\neq\emptyset$.
If $y\in C_\mathcal{B}(F)$, then $y=be$ for some $b\in B^1$,
$e\in E$ and $\mathcal{B}e=\mathcal{B}y\subseteq F=B^1\cdot
E$. Hence for every $\beta\in \mathcal{B}$ there exists $b\in
B^1$ and $e'\in E$ such that $\beta e=be'$ and so
$e=\beta^{-1}be\in\mathcal{B}E\cap E$ and so   $\beta^{-1}b=1_\mathcal{B}$ (due to $\mathcal{B}^\ast
E\cap E=\emptyset$, by Lemma 2.8) that means
$\beta=b\in B^1$. Therefore $B^1=\mathcal{B}$ which is a
contradiction. \\
Hence $\mathcal{B}u\nsubseteq F$ for every
$u\in F$, and so there exists $\beta\in
\mathcal{B}$ such that $\beta u\notin F$ and so $\beta u\notin A$ (for if $\beta u\in
A$ then $\beta u\in \mathcal{B}D$ so $u=\beta^{-1}\mathcal{B}D=\mathcal{B}D$ that is a contradiction).
This means $F\cap C_\mathcal{B}(A)=\emptyset$ and thus $C_\mathcal{B}(A)=\mathcal{B}D$ and $F=A\setminus \mathcal{B}D=A\setminus C_\mathcal{B}(A)$
(because $\mathcal{B}D\subseteq C\subseteq A=\mathcal{B} D\dot{\cup}F$).
\\
On the other hand, a calculation shows that
$$
A\setminus BA=(\mathcal{B} D\cup B^1 E)\setminus(\mathcal{B} D\cup B^2E\cup BE)=
(E\setminus BE)\setminus (\mathcal{B} D\cup B^2 E)=E\setminus (\mathcal{B} D\cup B^2 E)
$$
which requires $A\setminus BA\subseteq E$. Also,
$$
BA\cap E=(\mathcal{B} D\cap E)\cup (BE\cap E)\cup (B^2E\cap E)=B^2E\cap E\subseteq BE\cap E=\emptysetæ
$$
since $B\dot{\leq} S$ (by Theorem 3.5). Therefore $E=A\setminus BA$ and so
$$
BA\cap \mathcal{B}A^c=BA\setminus C_\mathcal{B}(A)\subseteq A\setminus C_\mathcal{B}(A)=B^1 E=B^1(A\setminus BA),
$$
and we arrive at $(2.4)$, since $(BA\cap \mathcal{B}A^c)\cap (A\setminus BA)=\emptyset$.\\
For proving that $D\dot{\cup} E$ is anti-left $\mathcal{B}^*$-transference, we have $D\cap\mathcal{B}E=\emptyset$, for if $x\in D\cap\mathcal{B}E$, then
$$
x=1_\mathcal{B}d=\beta e=\beta(\beta_ed_e)=(\beta\beta_e)d_e,
$$
where $d=x\in D$, $\beta_ed_e=e\in E$, $\beta,\beta_e\in \mathcal{B}$ and $d_e\in\mathcal{D}$. Thus,
$1_\mathcal{B}=\beta\beta_e$, $d_e=d\in D$ and so $e=\beta_ed\in E\cap \mathcal{B}D=\emptyset$ that is
a contradiction. Therefore, we have
$$
(D\cup E)\cap\mathcal{B}^\ast(D\cup E)=[(D\cap\mathcal{B}^\ast D)\cup(E\cap \mathcal{B}^\ast D)]
\cup  [(D\cap\mathcal{B}^\ast E)\cup(E\cap \mathcal{B}^\ast E)]=\emptyset,
$$
(note that $E\cap \mathcal{B}^\ast D\subseteq E\cap \mathcal{B} D\subseteq B^1E \cap \mathcal{B} D
= \emptyset$).
\end{proof}
{\bf Note.}
It may occur to the reader that the first requirement for the fundamental
theorem to give the direct representation for $A$ is that $B$ is a sub-semigroup of $S$,
and this makes the theorem ineffective for a large class of such subsets.
Therefore, we clarify that due to the equivalence of $BA\subseteq A$ and $\langle B\rangle A\subseteq A$
(by Proposition 3.2) it is sufficient to consider the sub-semigroup generated by $B$ instead of $B$ in the theorem.
For an important special case in groups, since $BA\subseteq A$ implies $\langle B\rangle A=BA$ (by Proposition 3.2) and
$\langle B\rangle^{-1}=\langle B^{-1}\rangle$, we have:
\begin{cor}
If $B$ is a non-empty subset of a group $G$ such that
$\langle \langle B\rangle\rangle=\langle B\rangle \dot{\cup} \langle B^{-1}\rangle \dot{\cup} \{1\}$, then
every left upper $B$-periodic subset $A$ has the unique direct representation
$$
A=\langle \langle B\rangle\rangle\cdot D\dot{\cup} \langle B\rangle^1\cdot E \;\;\; ; \;\;\; D\subseteq \mathcal{D},\; E\subseteq G
$$
if and only if
$$
BA\cap \langle \langle B\rangle\rangle A^c\subseteq \langle B\rangle^1(A\setminus BA)
$$
$($where $\mathcal{D}$ is the fixed right transversal of $ \langle \langle B\rangle\rangle$ in $G)$.
\end{cor}
As an example for the above corollary in the additive group of real numbers, if $B=\{1,2\}$, then the
initial condition for $B$ is satisfied, but for $B=\{2,3\}$ it does not work (although $B\cap -B=\emptyset$).\\
The following obtain some necessary and sufficient conditions for the above unique direct representations
to be $A=\mathcal{B}\cdot D$ (left $\mathcal{B}$-periodic case) or $A=B^1\cdot E$ (left $\mathcal{B}$-periodic free case,
see the next subsection).
It is a direct result of the fundamental theorem, considering the details of its proof (
for example, if $A=B^1\cdot E$, then $C_\mathcal{B}(A)=C_\mathcal{B}(F)=\emptyset$
and $(2.4)$ holds).

Under the conditions of Theorem 2.11 we have\\
$($a$)$
$$
A=B^1\cdot E \Leftrightarrow (2.4) \mbox{ and there is no any non-empty left $B$-periodic subset of $A$}
\Leftrightarrow (2.4) \mbox{ and } D=\emptyset
$$
$($b$)$
$$
A=\mathcal{B}\cdot D \Leftrightarrow (2.4) \mbox{ and $A$ is left $B$-periodic}
\Leftrightarrow (2.4) \mbox{ and } E=\emptyset \Leftrightarrow (2.4) \mbox{ and } A\setminus BA=\emptyset
$$

\begin{rem}
One of the important and basic results of the above main theorem is that
$D\dot{\cup} E$
is a minimal generator for the left upper $B$-periodic set $A$, in the sense that every member of $A$ is
 obtained with a left $\mathcal{B}$-shift of a unique element of $D\dot{\cup} E$
 and there is no proper subset with such a property.
 We will make important use of this fact in further research of this theory.\\
Also, it is worth noting that if we remove the condition $ D\subseteq \mathcal{D}$ for $(2.3)$
in Theorem 2.11 (only $D\subseteq S$), then all results of
the theorem is valid except the uniqueness of $D$.
\end{rem}
The next theorem not only is one of the most important achievements of the main theorem
 but also interestingly generalizes $(1.1)$ from $(\mathbb{R},+)$ to monoids and groups.
\begin{theorem}[The unique direct representation of left upper $b$-periodic subsets]
Let $b\in M$ be an invertible element such that $m=b^im$
if and only if $i=0$, for every $m\in M$ and $i\in \mathbb{Z}$. Then,
every left upper $b$-periodic subset $A$ has the unique direct representation
\begin{equation}
A=(b^\mathbb{Z}\cdot D)\dot{\cup} (b^{\mathbb{Z}_+^0}\cdot E),
\end{equation}
where $b^\mathbb{Z}:=\{ b^n:n\in \mathbb{Z}\}=\langle\langle b\rangle\rangle$, $b^{\mathbb{Z}_+^0}:=\{ b^n:n\in \mathbb{Z}_+^0\}=\langle b\rangle^1$,
$D\subseteq \mathcal{D}_b$, and $E\cap b^{\mathbb{Z}_+^0}\; E=\emptyset$.\\
In particular, if $M=G$ is a group and $b\in G$ is an element of infinite order, then $(2.5)$
holds.\\
$($The additive form of $(2.5)$ is $A=\mathbb{Z}b\dot{+} D\dot{\cup} {\mathbb{Z}_+^0}b\dot{+} E$.$)$
\end{theorem}
\begin{proof}
Putting $\mathcal{B}:=\langle\langle b\rangle\rangle$ we infer that $\mathcal{B}=b^\mathbb{Z}$
is a left factor subgroup of $M$ (and we fix it according Remark 1.4). Also,
setting $B:=\langle b\rangle$ we have $B=b^{\mathbb{Z}_+}$ and $\mathcal{B}=B\dot{\cup} B^{-1} \dot{\cup} \{1_M\}$.
For if $b^i=b^{-j}$ for some positive integers $i,j$, then $1_Mb^{i+j}=1_M$ and so $i+j=0$ that is a contradiction.
In addition, $\mathcal{B}_2=\{1_M\}$, since $\mathcal{B}$ can not have any nontrivial element of finite order.
Therefore. $B$ is a positive sub-semigroup of $\mathcal{B}$. Hence, it is enough (and necessary)
to show that $(2.4)$ holds. If $bA\subseteq A$ and $x\in  BA\cap \mathcal{B}A^c$, then
$x\in BA=bA$ and $x\notin (\mathcal{B}A^c)^c=\bigcap_{\beta\in \mathcal{B}}\beta^{-1} A$ (see subsection 2.1) thus
$x\notin \beta_0^{-1} A$ for some $\beta_0\in \mathcal{B}$. Thus, we conclude that
$\mathcal{B}x\nsubseteq A$ and $x=ba_0$,
for some $a_0\in A$. This implies that the set
$\{n\in \mathbb{Z}:b^na_0\notin A\}$ is nonempty and bounded above in integers, and so it
has the maximum element, namely $n_0$. It is clear that $n_0\leq-1$, $b^{n_0}a_0\notin A$,
and $b^{n_0+1}a_0\in A$. Hence, $b^{n_0+1}a_0\in A\setminus bA=A\setminus BA$ and so
$$
x=ba_0\in bb^{-n_0-1} (A\setminus BA)\subseteq B(A\setminus BA).
$$
Therefore, $BA\cap \mathcal{B}A^c \subseteq B(A\setminus BA)$ which requires  $(2.4)$.
Now, Theorem 2.11 completes the proof (note that if $M=G$ is a group, then $b$ satisfies the conditions if
and only if it has infinite order).
\end{proof}
\begin{ex}
Fix $b\in \mathbb{R}\setminus\{0\}$ and put $\mathcal{B}=b\mathbb{Z}$. Then $\mathcal{D}_b=\mathbb{R}_b:=b[0,1)$,
and $B:=b\mathbb{Z}_+$ is a positive subset of $\mathcal{B}$. Hence Theorem 2.14 requires that  every upper $B$-periodic set $A$
(in $(\mathbb{R},+)$) has  the unique direct
representation
$$
A=(b\mathbb{Z}\dot{+}bD)\dot{\cup}(b\mathbb{Z}_+^0\dot{+}E)\;\; ; \;\; D\subseteq [0,1), \; E\subseteq \mathbb{R}
$$
Note that all subsets of $\mathcal{D}_b$ (here) are of the form $bD$ where $D\subseteq [0,1)$. One can see
many interesting algebraic properties of $\mathbb{R}_b$ in \cite{grplike} and related references ($b$-parts of real numbers).
\end{ex}
\subsection{Periodic and upper periodic kernels of subsets.}
It is interesting to know that in an arbitrary magma, for every $B\subseteq X$, all subsets of $X$  contain the largest
$B$-periodic and upper $B$-periodic subsets. Because the empty set is $B$-periodic and the union of a family of left $B$-periodic
(resp. upper $B$-periodic) sets is so.
Therefore, we call the largest left $B$-periodic (resp. upper
$B$-periodic) subset of $A$ ``left $B$-periodic
(resp. upper $B$-periodic) kernel (or core)'' of $A$ and denote it by
$\Pk^\ell_B(A)$ (resp. $\Upk^\ell_B(A)$), and put $\Pf^\ell_B(A):=A\setminus\Pk^\ell_B(A)$. The right and two-sided cases are defined analogously.\\
These concepts are very useful and effective tools for studying periodic type sets.
Also, they have many connections and relations to other topics and play main roles in characterization, classification,
and unique direct representation of upper periodic subsets, sub-semigroups, and subgroups.\\
It is worth noting that as maps from $2^X$ into $2^X$ all of $\Pk^\ell_B$, $\Upk^\ell_B$, and
$\Pf^\ell_B$ are idempotent and their compositions have interesting properties.
 For brevity and according to the used notations above, in this paper, we put
$C_B=C(A)=C_B(A):=\Pk^\ell_B(A)$ and $F_B=F(A)=F_B(A):=\Pf^\ell_B(A)$.
Note that $C_B(A)=BC_B(A)\subseteq BA$ and so  $C_B(A)\subseteq A\cap BA$.\\
The both $Y=C_B(A)$ and $Y=\Upk^\ell_B(A)$ are solutions of the inequality $BY\subseteq A$.
But it is interesting to know that:\\
- Always $BY\subseteq A$ has the largest solution
$\Sigma^\ell_{A|B}:=\{x\in X|Bx\subseteq A\}$ (namely ``left $B$-summand of $A$'')  in $2^X$. \\
- $\Sigma^\ell_{A|B^1}$
is the largest solution of the inequality $BY\subseteq A$ in $2^A$.
Hence,
$$
C_B(A)\subseteq \Upk^\ell_B(A)\subseteq \Sigma^\ell_{A|B^1}=\Sigma^\ell_{A|B}\cap A\subseteq \Sigma^\ell_{A|B}.
$$
- In general, it does not need one of $A$ and $\Sigma^\ell_{A|B}$ be a subset of another one, but
there are some useful equivalent conditions for them (see \cite{ulpss}, pages 451-452). \\
- If $B\dot{\leq} S$, then
$$B\Sigma^\ell_{A|B^1}\subseteq B\Sigma^\ell_{A|B}\subseteq \Sigma^\ell_{A|B^1}\subseteq \Sigma^\ell_{A|B}$$
and so $\Sigma^\ell_{A|B}$ is left upper $B$-periodic
and $\Upk^\ell_B(A)=\Sigma^\ell_{A|B^1}$. \\
Thus if $B\dot{\leq}_\ell S$, then $C_B(A)=\Upk^\ell_B(A)=\Sigma^\ell_{A|B}$.\\
- If $A\dot{\leq} S$, then $\Sigma^\ell_{A|B^1}$ is either empty or a right ideal of $A$. For if $A\dot{\leq} S$, then
$$
(B\Sigma^\ell_{A|B})A\subseteq AA\subseteq A\Rightarrow B(\Sigma^\ell_{A|B}A)\subseteq A\Rightarrow
\Sigma^\ell_{A|B}A\subseteq \Sigma^\ell_{A|B}
$$
and so $\Sigma^\ell_{A|B}$ is a right ideal if and only if $\emptyset\neq \Sigma^\ell_{A|B}\subseteq A$. Also,
since always $\Sigma^\ell_{A|B^1}\subseteq A$,
$$
\Sigma^\ell_{A|B^1}A\subseteq (\Sigma^\ell_{A|B}A)\cap (AA)\subseteq \Sigma^\ell_{A|B}\cap A=\Sigma^\ell_{A|B^1}.
$$
- Therefore, if $B\dot{\leq}_\ell S$ and $A\dot{\leq} S$, then $C_B(A)$ is either empty or a right ideal of $A$.
This requires that {\bf if $B\dot{\leq}_\ell S$, then the left $B$-periodic kernel of every sub-semigroup is
either empty or a right ideal of it}. \\
- If $B\leq_\ell S$, then
\begin{equation}
C_B(A)=\Upk^\ell_B(A)=\bigcap_{b\in B}b^{-1} A=(BA^c)^c=\Sigma^\ell_{A|B}=\Sigma^\ell_{A|B^1}
\end{equation}
Hence, if $A$ is a sub-semigroup, then it has the useful right ideal $(2.6)$ with several descriptions and formulas.\\
In groups $G$, we also have $$\Sigma^\ell_{A|B}=\{x\in G|Ax^{-1}\cap B=B\},$$
and it is equal to  $\Sigma^\ell_{A|B}=\{x\in A|Ax^{-1}\cap B=B\}$ if $1\in B$ (this fact will be used in the last section).\\
It is worth noting that in semigroups the concept of
idealizator of a subset $A$ (defined by $Id\; A:=\{x\in S|xA\cup Ax\subseteq A\}$, see \cite{ideal1, ideal2}) is
a special case of the (two-sided) summand set when $B=A$, i.e.
$\Sigma_{A|A}$. Also,
$Sep\; A$ is defined by the meet of $Id\; A$ and $Id(S\setminus A)$. Hence the properties and results of this topic may be applied
for future study of idealizators and separators in semigroups.\\
As we know, $C_B(A)$ takes its most possible value (i.e., $A$) if and only if $A$ is
left $B$-periodic. In contrast to this property, when  $C_B(A)=\emptyset$ we say
$A$ is left $B$-periodic free. Note that only the empty set can be both left $B$-periodic and left $B$-periodic free.
Always, we can divide every subset into two components one of which is periodic
and the other is periodic free. For the purpose, fixing  a subset $B$, and for every subset $A$
we have
\begin{equation}
A=C_B\dot{\cup} F_B\; ,\; \Pk^\ell_B(C_B)=C_B\; , \; Pk^\ell_B(F_B)=\emptyset
\end{equation}
Because $C_B$ is left $B$-periodic and
$Pk^\ell_B(F_B)\subseteq Pk^\ell_B(A)\cap F_B=C_B\cap F_B=\emptyset$.
We use this fact for the classification of upper periodic subsets and also sub-structures (specially sub-semigroups).\\
The converse of $(2.7)$ is also true if $B\leq_\ell S$, i.e., if $A=C\dot{\cup} F$, $\Pk^\ell_B(C)=C$, and
$Pk^\ell_B(F)=\emptyset$, then $C=C_B(A)$ and $F=F_B(A)$.\\
For checking it, let $B,C$ satisfy the conditions $B\leq_\ell S$, $BC=C$, and $Pk^\ell_B(F)=\emptyset$.
Then $C\subseteq Pk^\ell_B(C\cup F)$. Now, if $x\in Pk^\ell_B(C\cup F)\cap F$, then $Bx\subseteq C\cup F$ and
$Bx\nsubseteq F$ (because $\Sigma^\ell_{F|B}=Pk^\ell_B(F)=\emptyset$), this means
$b_0x\notin F$, for some $b_0\in B$, and so  $b_0x\in C$ (since $b_0x\in Bx\subseteq C\cup F$)  which implies
$x\in b_0^{-1}C\subseteq BC=C$ that is a contradiction.
Therefore
$$Pk^\ell_B(C\cup F)=(Pk^\ell_B(C\cup F)\cap F)\cup(Pk^\ell_B(C\cup F)\cap C)\subseteq C.$$
Hence, it is obvious that under the above assumptions
$$
Pk^\ell_B(C\cup F)=C\; , \; Pf^\ell_B(C\cup F)=F\Leftrightarrow C\cap F=\emptyset,
$$
and we arrive at the next new lemma.
\begin{lem}[Existence and uniqueness of periodic and periodic free components of subsets] Fix a subset $B$ of $X$.\\
$($a$)$ For every set $A$ there exist subsets $C,F$ of $A$ such that $C$ is left $B$-periodic, $F$ is left $B$-periodic free,
and $A=C\dot{\cup} F$.\\
$($b$)$ If $(X=S$ is a semigroup and$)$ $B\leq_\ell S$, then the $($left$)$ periodic and periodic free components are unique $($and $C=C_B(A)$, $F=F_B(A))$.
\end{lem}
\textbf{$B$-Periodic classification of a family of subsets and sub-structures.}
One of the important achievements of the concept of  periodic kernels  is a periodic classification of
every family of subsets  of an algebraic
structure, especially sub-structures and upper periodic subsets.
Let $\mathcal{A}$ be a family of subsets of an algebraic structure $X$ and fix $B\subseteq X$. Then,
$\mathcal{A}$ is divided into three disjoint classes (not necessarily non-empty) as follows:\\
{\bf (1) The left $B$-periodic class}. The class of all $A\in \mathcal{A}$ such that $C_B(A)=A$. \\
{\bf (2) The left $B$-periodic free class.} The class of all $A\in \mathcal{A}$ such that is $C_B(A)=\emptyset$.\\
{\bf (3) The left $B$-mixed Class.} The class of all $A\in \mathcal{A}$ such that $\emptyset\subset C_B(A) \subset A$.\\
Right and two-sided classes are mentioned analogously.
\begin{ex}
Let $X=S$ be a semigroup and $\mathcal{A}$ the family of all sub-semigroups of $S$ containing a fixed subset $B$.
Then, all elements of $\mathcal{A}$ are (left) upper $B$-periodic. An element of $\mathcal{A}$ lies
in the first class if and only if it is left $B$-periodic. Hence, the members of the second and third classes are strictly
left upper $B$-periodic. Applying the classification for this family has a nice special case for real numbers
 that will be discussed in Section 4. For example, it is shown that
 the additive group of rationales does not contain any sub-semigroup of the  $\mathbb{Z}_+$-mixed class. Indeed,
 the $\mathbb{Z}_+$-periodic kernel of every upper 1-periodic sub-semigroup of rationales is either empty or
 the whole sub-semigroup.
   Also, in general, it is an idea for studying and classifying sub-semigroups of a semigroup
 or group containing a fixed subset (or element) which we hope to study in future papers.
\end{ex}
\subsection{The start of subsets and well started upper $B$-periodic sets}
Another important conception is the start of a subset relative to a fixed set.
Fix $B\subseteq X$ and for  every $A\subseteq X$ define
$\St^\ell_B(A):=A\setminus BA$ and call it left $B$-start of $A$.
The following is the reason for what we call $A\setminus BA$ left $B$-start of $A$.
If $A\subseteq S$ is left upper $B$-periodic and $b_r^{-1}$ exists, then $x\in \St^\ell_B(A)$ implies
$$\cdots,\; b_r^{-2}x\notin A,\; b_r^{-1}x\notin A,\; x\in A,\; bx\in A,\; b^2x\in A, \cdots  $$
(where  $b_r^{-1}$ is one of its right identities with respect to a left identity $l$). That means $x$ is the first element  of the sequence $$\cdots,\; b_r^{-2}x,\; b_r^{-1}x,\;x,\; bx,\;
b^2x, \cdots  $$ which belongs to $A$ (the start of the subsequence in $A$).\\
Similar to the periodic kernel, $\St^\ell_B:2^X\rightarrow 2^X$ is an idempotent map with interesting properties. For
example as composition of two maps we have $\St^\ell_B o\Pf^\ell_B=\St^\ell_B$ (because $BA=B(C_B\dot{\cup} F_B)=C_B\cup BF_B$
and so $A\setminus BA=F_B\setminus BF_B$) which implies the (left) start of the left periodic free component of a set
is the same as the (left) start of itself. \\
Similar to the previous notations, in this paper, we put $E_B=E(A)=E_B(A):=\St^\ell_B(A)$.
By the above argument, we always have
\begin{equation}
E_B(A)=E_B(E_B(A))=E_B(F_B(A))\subseteq F_B(A)
\end{equation}
It is obvious that
$$
E_B(A)=\emptyset \Leftrightarrow \mbox{ $A$ is left lower $B$-periodic}
$$
Thus a left upper $B$-periodic set is left $B$-periodic if and only if $E_B(A)=\emptyset$.\\
{\bf Characterization of anti-left $B$-transference subsets by the starts.}
One can see that for $E\subseteq X$ we have
$$
\mbox{ $E$ is anti-left $B$-transference} \Leftrightarrow \St^\ell_B(E)=E \Leftrightarrow E=E_B(A) \mbox{ for some $A\subseteq X$}
$$
Therefore, general form of all anti-left $B$-transference subsets $E$ of $X$ is $E=E_B(A)$ for all $A\subseteq X$.
This means the class of all anti-left $B$-transference subsets is in the same range of the map $\St^\ell_B:2^X\rightarrow 2^X$.
All sets of the form $\Sigma^\ell_{A|B}\setminus A$ ($A\subseteq X$) are anti-left $B$-transference. Because
$$
(\Sigma^\ell_{A|B}\setminus A)\cap B(\Sigma^\ell_{A|B}\setminus A)\subseteq (\Sigma^\ell_{A|B}\setminus A)\cap B\Sigma^\ell_{A|B}
\subseteq (\Sigma^\ell_{A|B}\setminus A)\cap A=\emptyset
$$
and so $E_B(\Sigma^\ell_{A|B}\setminus A)=\Sigma^\ell_{A|B}\setminus A$.
It is interesting to know that  $\Sigma^\ell_{A|B}\setminus A$ gives us the general form of all anti-left $B$-transference
subsets
if $B$ has the left summand property, that is, $\Sigma^\ell_{BY|B}=B^1Y$ for all $Y\subseteq X$. For example, it is the general form
if $B=\langle b\rangle$ and $b$ has the left cancelation property. We also have the general form $E_B(\Sigma^\ell_{A|B})$ if
$B$ is a sub-semigroup with the left summand property.
For more details see Lemma 4.2, 4.3, 4.5 from \cite{ulpss}.\\
In Theorem 2.11 we have $E_B(F_\mathcal{B}(A))=E_B(A)$, since putting $F:=F_\mathcal{B}(A)$
we have
$$
F\setminus BF=F\setminus (B^2\cup B)E=F\setminus BE=E=A\setminus BA.
$$
Note that here $B=\mathcal{B}_+$ and this is different to $(2.8)$.
Since the condition $(2.4)$ is necessary and sufficient for the unique direct representation of
(left) upper periodic subsets and due to $(2.6)$, we state the following useful definition.
\begin{defn}
Let $B,\mathbb{B}$ be fixed subsets of $X$.  We call a left upper $B$- periodic subset $A$ well left $(\mathbb{B},B)$-started if
$BA\setminus C_\mathbb{B}(A) \subseteq BE_B(A)$.
\end{defn}
- First note that $A$ is well left $(\mathbb{B},B)$-started if and only if
  $BA\setminus C_\mathbb{B}(A) \subseteq B^1E_B(A)$.\\
- Since $C_\mathbb{B}(A)\subseteq \Upk^\ell_\mathbb{B}(A)\subseteq \Sigma^\ell_{A|\mathbb{B}}$, if $A$ is
well left $(\mathbb{B},B)$-started, then $BA\setminus \Sigma^\ell_{A|\mathbb{B}}(A) \subseteq BE_B(A)$.
but for the converse we have the following challenges.  \\
{\bf Question I.} Are there counter examples for  $BA\setminus \Sigma^\ell_{A|\mathbb{B}}(A) \subseteq BE_B(A)\Rightarrow BA\setminus C_\mathbb{B}(A) \subseteq BE_B(A)$, $BA\setminus \Upk^\ell_\mathbb{B}(A) \subseteq BE_B(A)\Rightarrow BA\setminus C_\mathbb{B}(A) \subseteq BE_B(A)$, and $BA\setminus \Sigma^\ell_{A|\mathbb{B}}(A) \subseteq BE_B(A)\Rightarrow BA\setminus \Upk^\ell_\mathbb{B}(A) \subseteq BE_B(A)$?\\
Also, find some conditions in semigroups that imply all the above implications are valid (except the condition $B\dot{\leq}_\ell S$,
because it requires $C_B(A)=\Upk^\ell_B(A)=\Sigma^\ell_{A|B}$).
 \\
- If  $A$ is either left $B$-periodic left $\mathbb{B}$-periodic, then we have the following determining the situation for
 well left $(\mathbb{B},B)$-starting of $A$.\\
Indeed, if $BA=A$, then $A$ is well left $(\mathbb{B},B)$-started  if and only if
$\mathbb{B}A=A$. Because it is equivalent to $A\setminus C_\mathbb{B}(A) \subseteq B^1\emptyset=\emptyset$,
which is $A \subseteq C_\mathbb{B}(A)$. This means  $A=C_\mathbb{B}(A)$ since always  $C_\mathbb{B}(A) \subseteq A$.
Hence, if $BA=A=\mathbb{B}A$, then $A$ is well left $(\mathbb{B},B)$-started.
For example, since $\mathbb{Z}_++A=A$ if and only if $\mathbb{Z}+A= A$, every $\mathbb{Z}_+$-periodic set $A$ in real numbers is
well left $(\mathbb{Z},\mathbb{Z}_+)$-started (as we know).\\
 But, if $\mathbb{B}A=A$, then every left upper $B$-periodic subset $A$ is well left $(\mathbb{B},B)$-started,
 for all subsets $B$ (because it is equivalent to $\emptyset =BA\setminus A \subseteq B^1E_B(A)$).  \\
 - For the case $\mathbb{B}=B$, ``$A$ is well left $(B,B)$-started'' is equivalent to $BA\setminus C_B(A) = BE_B(A)$
if $F_B(A)$ is left upper $B$-periodic (equivalently $BF_B(A)\cap C_B(A)=\emptyset$).  \\
- If a left upper $B$ periodic subset $A$ is well left $(\mathbb{B},B)$-started,
 then ($\langle B\rangle A=BA$ and so) it is well
left $(\mathbb{B},\langle B\rangle)$-started,
but the converse is not true (e.g., consider $B=\{1\}$ and $A=(0,+\infty)$ in the additive group of real numbers).
Also, in groups:
$$
\mbox{$A$ is well left $(\langle\langle B\rangle\rangle,\langle B\rangle)$-started } \Leftrightarrow
BA\cap \langle \langle B\rangle\rangle A^c\subseteq \langle B\rangle E_B(A))
$$
and $\langle B\rangle$ in the right side can not be replaced by $B$ (as we used it in Corollary 2.12).\\
Now, we give some special classes of subsets satisfying the well left $(\mathbb{B},B)$-started property,
and also some counter-examples for it. Then, we construct a vast class of left upper $B$-periodic subsets $A$ with the
property.
\begin{ex}
If $b$ is an invertible element of the monoid $M$, then
 putting $B:=\langle b\rangle=\{ b^n|n\in \mathbb{Z}^+ \}$ and $\mathbb{B}=\langle\langle b\rangle\rangle$, every
left upper $b$-periodic subset $A$ is  well left $(\mathbb{B},B)$-started (or simply well left $b$-started), see the proof of Theorem 2.14.\\
Now, consider the additive group $(S=\mathbb{R},+)$ and put
$B=\mathbb{Q}^+$, $\mathcal{B}=\mathbb{Q}$, $A=[M,+\infty)$.
Then, all conditions of Theorem 2.11 except that the well $(\mathcal{B},B)$-started property
are satisfied. Because
$$B+A=\mathbb{Q}^++[M,+\infty)=(M,+\infty)\; , \;
C_\mathcal{B}(A)=\Sigma_{A|\mathcal{B}}=\{x\in \mathbb{R}|\mathbb{Q}+x\subseteq [M,+\infty)\}=\emptyset ,$$
$$\mbox{St}_B(A)=[M,+\infty)\setminus (M,+\infty)=\{ M\}.$$
Therefore, $(B+A)\setminus C_\mathcal{B}(A)=(M,+\infty)$ that is not a subset of
$B+\mbox{St}_B(A)=\mathbb{Q}^++M$ (indeed, we have
$B+\mbox{St}_B(A)\subset (B+A)\setminus\Sigma_{A|\mathcal{B}}$).
\end{ex}
\begin{rem}
Let $B$ and $\mathbb{B}$ are sub-semigroups of $S$ for which
$\mathbb{B}$ is left $B$-periodic. Then putting  $A=\mathbb{B}D\cup B^1E$, for all
subsets $D$ and $E$, we have:\\
(i) $BA=\mathbb{B}D\cup BE$ so $A$ is left upper $B$-periodic.\\
(ii) $E_B(A)=E\setminus BA\subseteq E$.\\
(iii) $\mathbb{B}D\subseteq
\Sigma^\ell_{\mathbb{B}D|\mathbb{B}}\subseteq
\Sigma^\ell_{A|\mathbb{B}}$ so $BA\setminus
\Sigma^\ell_{A|\mathbb{B}}=BE\setminus
\Sigma^\ell_{A|\mathbb{B}}\subseteq BE$\\
Therefore if $E\cap BE=E\cap \mathbb{B}D=\emptyset$, then $E_B(A)=E$ so
$BA\setminus \Sigma^\ell_{A|\mathbb{B}} \subseteq BE_B(A)$. Hence, we conclude that $A$ is well left $(\mathbb{B},B)$-started
if $\Sigma^\ell_{A|\mathbb{B}}=C_\mathbb{B}(A)$ (e.g., if $\mathbb{B}\leq_\ell S$). \\
In particular, if $B\dot{\leq}\mathbb{B}\leq_\ell S$, $A=\mathbb{B}D\cup B^1E$, and  $E\cap BE=E\cap \mathbb{B}D=\emptyset$,
then ($BA\subseteq A$ and)  $A$ is well left $(\mathbb{B},B)$-started.\\
{\bf Question II.} Let $A=\mathbb{B}D\cup B^1E$, $B\dot{\leq} S$, $\mathbb{B}\dot{\leq} S$, and $B\mathbb{B}=\mathbb{B}$. Then
is $A$ well left $(\mathbb{B},B)$-started? (we have $BA\setminus \Sigma^\ell_{A|\mathbb{B}} \subseteq BE_B(A)$
according the above assertion)\\
Also, can one conclude $C_\mathbb{B}(A)=\mathbb{B}D$ from the above last assumptions
($B\dot{\leq}\mathbb{B}\leq_\ell S$, $A=\mathbb{B}D\cup B^1E$, and  $E\cap BE=E\cap \mathbb{B}D=\emptyset$)?\\
Therefore, Theorem 2.11 together with the above explanations implies that if $\mathcal{B}_2=\{1_\mathcal{B}\}$ and $\mathcal{B}_+\neq\emptyset$,
then every well left $(\mathcal{B},\mathcal{B}_+)$-started upper $\mathcal{B}_+$-periodic set $A$ has the direct representation
\begin{equation}
A=\mathcal{B}\cdot(C_\mathcal{B}(A)\cap
\mathcal{D})\dot{\bigcup} \mathcal{B}_+^1\cdot E_{\mathcal{B}_+}(A)
\end{equation}
Moreover, every similar representation $A=\mathcal{B}\cdot D\dot{\cup} B^1\cdot E$ implies
$E=E_{\mathcal{B}_+}(A)$, and also $D=C_\mathcal{B}(A)\cap
\mathcal{D}$ if the (initial) condition $D\subseteq \mathcal{D}$ holds (which means the direct representation
is completely unique).
\end{rem}
\subsection{Topology of upper periodic subsets and related topological groups and semigroups}
There are topological aspects of the periodic type classes of algebraic structures some of
 which are topological magmas, semigroups, and groups. Its important point is that the related
topologies, through the topic, come from the binary operation of the ground algebraic structure,
 and so their properties are strongly related to the algebraic characteristics.\\
Now, fix $B\subseteq X$ and put
$$\tau^\ell_{B-p}(X)=\tau^\ell_{B-p}:=\{A\subseteq X|A\; \mbox{is left $B$-periodic}\}
$$
$$
\tau^\ell_{B-up}(X)=\tau^\ell_{B-up}:=\{A\subseteq X|A\; \mbox{is left upper $B$-periodic}\}
$$
$$
\tau^\ell_{B-lp}(X)=\tau^\ell_{B-lp}:=\{A\subseteq X|A\; \mbox{is left lower $B$-periodic}\}
$$
Because every union and intersection of left
upper $B$-periodic sets are so,
$\tau^\ell_{B-up}$ is a topology in $X$ and so we call it {\em
left upper $B$-periodic topology}. In this topology, the
intersection of every family of open sets is open, which means $\tau^\ell_{B-up}(X)$ is an ``Alexandrov topology'' in $X$
(see \cite{Alx}). In particular, putting $B=X$, we obtain $\tau^\ell_{X-up}(X)$ that is the topology of all left ideals and
empty set. Hence, we arrive at the next theorem.
\begin{theorem} In every magma $X$ we have\\
$($a$)$ Always $\tau^\ell_{B-up}$ is an Alexandrov topology in $X$.\\
$($b$)$ If $B$ has the property that for every subset $A$ the equality $BA=A$
requires $bA=B$, for all $b\in B$ $($see Remark 3.3$($a$))$ and $X$ is left $B$-cancelative, then $\tau^\ell_{B-p}$ is an Alexandrov topology in $X$.\\
$($c$)$ If $X\subseteq BX$ $($equivalently $X=BX)$ and $B$ has the property that for every subset $A$ the inequality $BA\supseteq A$
requires $bA\supseteq B$, for all $b\in B$ $($see Remark 3.3$)$ and $X$ is left $B$-cancelative, then $\tau^\ell_{B-lp}$ is an Alexandrov topology in $X$.\\
$($d$)$ If $\tau^\ell_{B-p}$ is a topology in $X$, then $\tau^\ell_{B-up}$ is finer than $\tau^\ell_{B-p}$.
\end{theorem}
\begin{proof}
Let $A_i$ be a family of subsets of $X$ and $B\subseteq X$. Then, we have
\begin{equation}
B(\bigcup_iA_i)=\bigcup_iBA_i\;\; , \;\; B(\bigcap_iA_i)\subseteq
\bigcap_iBA_i
\end{equation}
Note that if $B(\bigcup_iA_i)=B\cdot (\bigcup_iA_i)$,
then $B(\bigcap_iA_i)= B\cdot (\bigcap_iA_i)=\bigcap_iB\cdot A_i=\bigcap_iBA_i$,
which means (in that case) for the intersection also the equality holds.
Especially if $b$ has the left cancellation property, then
$b(\bigcap_iA_i)=\bigcap_ibA_i$. Now one can obtain the results, easily.
\end{proof}
\begin{ex}
If $S$ is left $B$-cancellative and $B$ satisfies the conditions of Remark 3.3 (a), then the  both
$\tau^\ell_{B-up}$ and $\tau^\ell_{B-p}$ are topologies in $S$.
Note that $\tau^\ell_{B-up}=\tau^\ell_{B-p}$ if $B$ contains a left identity of $S$. In particular,
if $B\leq G$, then  $\tau^\ell_{B-up}=\tau^\ell_{B-p}$ is a topology in $G$.
\end{ex}
The next theorem shows that an upper periodic topology may form a topological semigroup or group.
\begin{theorem} Put $\tau:=\tau^\ell_{B-up}(S)$.\\
$($a$)$ If $B\dot{\leq} S$ and $xB\subseteq Bx$ for all $x\in S$ $($i.e., $B$ is a left normal sub-semigroup of $S)$,
then $S$ together with   $\tau$ is a topological semigroup and $\{B^1x:x\in S\}$
$($i.e., right co-sets of $B^1$ in $S)$ is a basis for the Alexandrov topology $\tau$. \\
$($b$)$ With the assumptions of $($a$)$, if $S=G$ is a group, then $(G,\cdot, \tau)$ is a topological group
if and only if $B\unlhd G$. And if this is the case, then the basis is equal to $G/B$ $($i.e., the ground set of the quotient group$)$.
\end{theorem}
\begin{proof}
First note that the condition $xB\subseteq Bx$, for all $x\in S$, is equivalent to $AB\subseteq BA$, for
all $A\subseteq S$. Also,  $\{B^1x:x\in S\}$ is a basis for $\tau$, because
$$
B(B^1x)=(BB)x\cup Bx=Bx\subseteq B^1x,
$$
since $B\dot{\leq} S$. Now, let $x,y\in S$ and $U_{xy}$ be a neighborhood of $xy$ and
put $U_x=B^1x$, $U_y=B^1y$. Then,
$$
xy\in U_xU_y=(B^1B^1)xy=B^1(xy)\subseteq B^1U_{xy}=U_{xy}
$$
Thus, the first part is proved. \\
For part (b), first note that the topological semigroup $(G,\cdot, \tau)$ is a topological
group if and only if $B^{-1}$ is open which means $BB^{-1}\subseteq B^{-1}$
or equivalently $B\leq G$. Also, a subgroup of $G$ is left normal if and
only if it is normal.
\end{proof}
\begin{cor}
Let $b\in S$ has the property that for every $x\in S$ and $n\in \mathbb{Z}_+$ there exists
$m\in \mathbb{Z}_+$ such that $xb^n=b^mx$ $($e.g., if $b$ is a central element$)$. Then,
$(S,\cdot, \tau^\ell_{b-up}(S))$ is a topological semigroup
$($note that $\tau^\ell_{b-up}(S):=\tau^\ell_{\{b\}-up}(S)=\tau^\ell_{\langle b\rangle-up}(S))$.
Also, $\tau^\ell_{\langle\langle b\rangle\rangle-p}(G)$ makes $G$ a topological group if  the
cyclic group generated by $b$ is normal in $G$.
\end{cor}
\begin{ex}
Putting $\tau:=\{A\subseteq \mathbb{R}: 1+A=A\}$, we have the topological group
$(\mathbb{R},+,\tau)$ (since $1+A=A$ if and only if $\mathbb{Z}+A\subseteq A$,
and $\mathbb{Z}$ is a normal subgroup). Note that $$\tau=\{\mathbb{Z}+D: D\subseteq [0,1)\}.$$
Also $\tau':=\{A\subseteq \mathbb{R}: 1+A\subseteq A\}$ provide  the topological semigroup
$(\mathbb{R},+,\tau')$ with $$\tau'=\{(\mathbb{Z}+D)\dot{\cup}(\mathbb{Z}_+^0+E): D\subseteq [0,1),
(E-E)\cap \mathbb{Z}=\{0\}\}.$$
\end{ex}
It is worth noting that the condition of part (a) of the above Theorem 2.23 requires
$$\tau_{B-up}(S)=\tau^\ell_{B-up}(S)\subseteq \tau^r_{B-up}(S),$$
sice, under the condition, $BA\subseteq A$ implies $BA\cup AB\subseteq A$. Also, similar results can be obtained if $\tau^\ell_{B-p}(S)$
is also a topology (like the part (b) that we have $\tau^\ell_{B-p}(G)=\tau^\ell_{B-up}(G)$).
Anyway, there is a need for extensive study and research in this field,
some of which we will describe below.\\
{\bf Project III.} A study on the periodic-topological topic that covers the following seems to be useful:\\
(1) Since in every finite group $G$ we have $\tau^\ell_{B-up}(G)=\tau^\ell_{B-p}(G)$, for all $\emptyset\neq B\subseteq G$,
study of this topology and related topological group is of special interest.  \\
(2) (Relations to the topology of left ideals and empty set) Putting
$$
\tau^\ell_{up}(X)=\tau^\ell_{up}:=\bigcap_{B\in P(X)}\tau^\ell_{B-up}
$$
we have
$$
A\in \tau^\ell_{up}\Leftrightarrow  BA\subseteq A\; ;\; \forall B\subseteq X\Leftrightarrow  XA\subseteq A
\Leftrightarrow \mbox{ $A$ is either empty or a left ideal of $X$}
$$
Therefore, the intersection of all left upper $B$-periodic topologies is indeed the topology of left ideals and empty set.
Hence $\tau^\ell_{up}$ is an Alexandrov topology and $(S,\tau^\ell_{up})$ may be a topological semigroup (under some conditions).
Thus, it needs more studies.
\\
(3) For a semigroup $S$ we have
$\tau^\ell_{B-up}\dot{\leq}\mathcal{P}(S)$, where $(\mathcal{P}(S),\cdot)$ is the power semigroup of $S$.
Indeed, $\tau^\ell_{B-up}$ is a right ideal
of  $(\mathcal{P}(S),\cdot)$. Hence, a study of this structure should be useful and notable. \\
(4) If $(S,\cdot, \tau^\ell_{B-p}(S))$ is a topological semigroup (resp. group), then one may obtain some properties or conditions for $B$.
\\
{\bf Project IV.} As an important result and achievement of the topic we can solve or study
many equations, identities, or inequalities (inclusions) in the power semigroup $\mathcal{P}(S)$  some of which are as follows:
$$
BY=A \;, \;BY\subseteq A \;, \; B^1Y\subseteq Y\;, \; BY\subseteq Z\subseteq A,\; \cdots
$$
where $A,B, \cdots$ are known elements of $\mathcal{P}(S)$ and $Y,Z, \cdots$ are unknown (variables).\\
Now, according to the concepts, tools, and results we have obtained,
it is possible for us to study and examine such relationships. For example:\\
(1) General solution of $BY\subseteq A$ (the set of all solutions) is $\mathcal{P}(\Sigma^\ell_{A|B})$, and $\Sigma^\ell_{A|B}$ is the largest solution.\\
(2) If $B\leq_\ell S$, then $\mathcal{P}((BA^c)^c)$ is the explicit general solution of $BY\subseteq A$ (that is equal to $\mathcal{P}(\Sigma^\ell_{A|B})$).\\
(3) In all the above cases, the general solution of $B^1Y\subseteq A$ is the intersection of the general solution of $BY\subseteq A$
and $\mathcal{P}(A)$.\\
(4) General solution of $BY\subseteq Y\subseteq A$ in $\mathcal{P}(S)$ is
$$\mathcal{P}(\Upk^\ell_B(A))\cap\tau^\ell_{B-up}(S)=\mathcal{P}(\Upk^\ell_B(A))\cap\{\langle B\rangle^1 E: E\subseteq S\}.$$
(5) If $B\dot{\leq} \mathbb{B}\leq_\ell S$, then $(\mathbb{B}D\cup BY)\dot{\cup}Y=A$ has a unique solution
in the power semigroup $\mathcal{P}(S)$ (see Observation 1 in Subsection 3.3). For example, the equation $(0,M)=(0,\frac{1}{2}]Y\dot{\cup} Y$ has the unique solution
$Y=[\frac{M}{2}, M)$ in the power semigroup of $(\mathbb{R},\cdot)$. \\
In algebraic structures with more than one binary operation,
we will encounter a system of equations, inequalities, or a combination of the two. As an example, if $(R,+,\cdot)$
is a ring or field, then the above results can be applied to many systems of equations and inequalities such as
$$
\left\{\begin{array}{l}
 B+Y=A  \\
BY=A
\end{array} \right.
\;\; , \;\;
\left\{\begin{array}{l}
 B+Y\subseteq A  \\
BY\subseteq A
\end{array} \right.
\;\; , \;\;
\left\{\begin{array}{l}
 B+Y= A  \\
BY\subseteq A
\end{array} \right.
\; , \cdots
$$
It is worth noting that $\emptyset\neq I\subseteq R$ is a (two-sided) ideal of $(R,+,\cdot)$ if and only if $Y=I$ is a symmetric solution of the system
of inequalities
 $$
\left\{\begin{array}{l}
 I+Y\subseteq I  \\
RY\subseteq Y\\
YR\subseteq Y
\end{array} \right.
$$
This is equivalent to the properties: $I$ is upper $I$-periodic in $(R,+)$, upper $R$-periodic in $(R,\cdot)$, and $I=-I$.
Therefore, the classification, general solutions, and other features of such inequalities and equations require extensive research.
\section{General form, directness and uniqueness of representations of upper periodic subsets}
 As the reader has noticed, three essential items for upper periodic sets have been presented:
 general form, direct representation, and their uniqueness.
The main theorem gives all the three items if $B$ is a positive subs-semigroup of
a left factor sub-group $\mathcal{B}$.
  But if we want to realize some of the above three goals, we can cover a wider range of
  algebraic structures and subsets of periodic type.
The following basic propositions contain some important basic results for periodic type subsets that are the key to
develop the theory.
\begin{prop} The following basic properties hold in every magma:\\
$($a$)$ $A$ is left upper $B$-periodic if and only if it
is left $B^1$-periodic as a subset of $X^1$, and if and only if $A$ is
left upper $b$-periodic $($resp. $B'$-periodic$)$ for every $b\in B$ $($resp. $B'\subseteq B$$)$.\\
$($b$)$ If $A$ is left $b$-periodic $($resp. lower $b$-periodic$)$  for every $b\in B$, then
$A$ is left $B$-periodic $($resp. lower $B$-periodic$)$, but the converse is not true .\\
$($c$)$ If $X$ contains a left identity $l$ and $l\in B$, then every
subset of $X$ is lower $B$-periodic.\\
$($d$)$ If $A$ is a sub-magma of $X$, then $A$ is
upper $B$-periodic, for every $B\subseteq A$.\\
$($e$)$ Considering $B=X$, a nonempty subset of $X$ is left upper
$X$-periodic if and only if it is a left ideal of $X$.\\
\end{prop}
\begin{proof}
This is easy to prove.
\end{proof}
\begin{prop} The following basic properties hold in every semigroup:\\
$($a$)$ $A$ is left upper $B$-periodic if and only if it is left upper
$\langle B\rangle$-periodic, and we have
$$
BA\subseteq A\Leftrightarrow B^1A=A\Leftrightarrow \langle B\rangle A\subseteq A \Leftrightarrow
\langle B\rangle^1A=A\Leftrightarrow bA\subseteq A\;; \; \forall b\in B
$$
Moreover, each of the above equivalent conditions requires $BA=\langle B\rangle A$ and so
$\St_B^\ell(A)=\St_{\langle B\rangle}^\ell(A)$.\\
$($b$)$ $A$ is left $B$-periodic if and only if it is left
$\langle B\rangle$-periodic.
$($c$)$ If  $\langle \langle B\rangle\rangle\leq_\ell S$, then
$$
\langle \langle B\rangle\rangle A=A \Leftrightarrow bA=A\;; \; \forall b\in B
\Rightarrow BA=A
$$
hence, $A$ is left $\langle \langle B\rangle\rangle$-periodic implies
$A$ is left $B$-periodic, but the converse is not true.
\\
$($d$)$
If $A$ is left lower $B$-periodic, then $\langle B\rangle A$ is left lower
$\langle B\rangle$-periodic, and we have
$$
A\subseteq BA\Leftrightarrow BA=B^1A \Rightarrow A\subseteq  \langle B\rangle A\Rightarrow \langle B\rangle A=\langle B\rangle B A=\langle B\rangle^2 A.
$$
Hence, if $B\dot{\leq} S$, then $A$ is left lower $B$-periodic implies $BA$ is left $B$-periodic.
\end{prop}
\begin{proof}
If $BA\subseteq A$ (resp. $BA=A$), then $B^n A\subseteq BA$ (resp. $B^nA=BA$) for all $n\in \mathbb{Z}_+$.
Hence,
$$
BA\subseteq \langle B\rangle A=(\cup_{n\geq 1} B^n)A=\cup_{n\geq 1}B^nA\subseteq BA
$$
Thus $\langle B\rangle A=BA\subseteq A$ (resp. $\langle B\rangle A=BA=A$).\\
Conversely, if $\langle B\rangle A\subseteq A$ (resp. $\langle B\rangle A= A$),
then $B A\subseteq \langle B\rangle A\subseteq A$ (resp. $BA\subseteq A$ and so $A=\langle B\rangle A=BA$). This
completes the proof of (a) and (b).\\
A generalization of (c) is proved in the next remark. Note that if $S=\mathbb{R}$ is the additive group of real numbers,
$x$ is an irrational number, $B:=\{1,x\}$, and $A:=\langle \{1,x,-1\}\rangle \cup\{0\}$, then $B+A=A$ but $\langle\langle B\rangle\rangle+A\neq A$
(since $-x\notin A$).
\\
Now, let $A\subseteq BA$. Since $B\subseteq  \langle B\rangle$ and $\langle B\rangle B\subseteq  \langle B\rangle^2 \subseteq  \langle B\rangle$,
we obtain
$$
 \langle B\rangle A\subseteq \langle B\rangle B A\subseteq  \langle B\rangle^2 A\subseteq  \langle B\rangle A
$$
This proves (d) clearly.
\end{proof}
\begin{rem}
As we saw before, for upper periodicity, always, $BA\subseteq A$ implies $bA\subseteq A$ for all $b\in B$.
But, in general, similar properties do not hold for periodicity and lower periodicity.
The previous proposition mentions that under some conditions on $B$, it should be valid for every left $B$-periodic subset $A$.
We can generalize some of the results as follows:\\
(a) If $B_\ell^{-1}(l)\subseteq B\subseteq S$ and $bb_\ell^{-1}=l$, for every
$b\in B$ (e.g., if  $B^{-1}=B\subseteq M$ or $B\leq G$), then
$$
BA=A\Rightarrow bA=A\; ;\; \forall b\in B
$$
For if $a\in A$, then $a=la=bb^{-1}a\in bB^{-1}A\subseteq bBA=bA$, for all $b\in B$.
Thus $A\subseteq bA\subseteq A$ and so $bA=A$ for all $b\in B$.\\
Note that the condition $B^{-1}=B\subseteq M$ does not work for lower periodicity. For example,
if $B=\{-1,0,1\}$ and $A=\{0,1\}$ in $(\mathbb{R},+)$, then $A\subseteq B+A$ but $A\nsubseteq 1+A$
(for $b=1\in B$). In general, only for the
trivial case $B=\{b\}$, we have the property for lower periodicity.
Although, we have the following related properties:\\
(b) If $bA\subseteq A$ (resp. $A\subseteq bA$)
and $b^{-1}=b_\ell^{-1}(l)$ exists, then $A\subseteq b^{-1}A$ (resp. $b^{-1}A\subseteq A$).\\
Also, if $BA\subseteq A$ and $B\cap B^{-1}\neq\emptyset$ (i.e., $B$ is not anti-inversion), then
$(B\cap B^{-1})A=A$ and so $BA=A$, $A\subseteq B^{-1}A$.\\
Because $(B\cap B^{-1})A\subseteq BA\subseteq A$, and  if $a\in A$ and $b_0\in B\cap B^{-1}\neq\emptyset$, then
$$a=la=b_0^{-1}(b_0a)\subseteq b_0^{-1}BA\subseteq b_0^{-1}A\subseteq (B\cap B^{-1})A,$$
Hence, $(B\cap B^{-1})A=A$ and one conclude that $BA=A$, $A\subseteq B^{-1}A$.\\
\end{rem}
{\bf Project V.} In every semigroup $S$ (resp. group $G$), consider the following properties for subsets $A,B$:\\
($P_1$) $\langle B\rangle A=A$ (i.e., $A$ is left $\langle B\rangle$-periodic),\\
($P_2$) $BA=A$ (i.e., $A$ is left $B$-periodic),\\
($P_3$) $BA=B^{-1}A=A$ (i.e., $A$ is left $B$ and $B^{-1}$-periodic),\\
($P_4$) $(B\cup B^{-1})A=A$ (i.e., $A$ is left $B\cup B^{-1}$-periodic),\\
($P_5$) $bA=A$ for all $b\in B$ (i.e., $A$ is completely left $B$-periodic),\\
($P_6$) $\langle\langle B\rangle\rangle A=A$ (i.e., $A$ is left $\langle \langle B\rangle\rangle$-periodic, if $\langle\langle B\rangle\rangle$ exists).\\
By Proposition 3.1, 3.2, and Remark 3.3, $P_1$ and $P_2$ are equivalent (in semigroups), and
$P_3$ till $P_6$ are equivalent if  $\langle \langle B\rangle\rangle\leq_\ell S$. Note that $(B\cup B^{-1})A=A$ if and only if
$\langle\langle B\rangle\rangle A=\langle B\cup B^{-1}\rangle A=A$.
Also,  $P_2\nRightarrow P_3$ (see the end of the proof of
Proposition 3.2) and $P_2\nRightarrow P_5$ (even we have an example that $BA=A$ but $bA\neq A$ for all $b\in B$, e.g.,
let $A=(0,+\infty)$ and $B=(0,1)$ in $(\mathbb{R},+)$).\\
Therefore, it is a useful and interesting project to study relations between them
some of which  as follows:\\
(1) Find some sufficient or/and necessary conditions on $B$ such that all (resp. some) of the above conditions
are equivalent.\\
For example, in groups, all of them are equivalent (for every $A\subseteq G$), if $B$ is symmetric.\\
(2) Do similar study  under the condition $A\dot{\leq} S$ (resp. $A\dot{\leq} G$).\\
(3) Do more studies in special numerical groups and semigroups such as in integers, rationales, reals, and complex numbers.\\
For example, by using the above facts and a little more review, we find that
$$
1+A= A\Leftrightarrow \langle1\rangle+A=A \Leftrightarrow \langle\langle1\rangle\rangle+A= A\Leftrightarrow \{1,2\}+A=A
\Leftrightarrow \langle2,3\rangle+A=A$$
Now, e.g., are they equivalent to $\{2,5\}+A=A$?\\
More generally, find some sufficient or/and necessary conditions
on $\{b,\beta\}$ (resp. $B$) such that $\{b,\beta\}+A=A$ (resp. $B+A=A$) is equivalent to $1+A=A$.
Our guess is that $|b-\beta|=1$ or $\gcd(b,\beta)=1$. Also, does $\langle\langle B\rangle\rangle+A=A$ imply
 $B\subseteq \mathbb{Q}$ or $\mathbb{Z}$?\\
(4) Study relations between the representations of a periodic (resp. an upper periodic) set $A$ respect to different periods (resp. upper periods),
especially when $A$ is a sub-semigroup.\\
For example, if $A\subseteq \mathbb{R}$, $b_1b_2\neq 0$, and $A=b_1+A=b_2+A$, then we have
$$
A=(b_1\mathbb{Z}\dot{+}b_1D_1)\dot{\cup}(b_1\mathbb{Z}_+^0\dot{+}b_1E_1)=(b_2\mathbb{Z}\dot{+}b_2D_2)\dot{\cup}(b_2\mathbb{Z}_+^0\dot{+}b_2E_2),
$$
with the mentioned conditions. Now, what relations are held between $D_1$ and $D_2$ (resp. $E_1$ and $E_2$)?
\subsection{General form of left periodic type subsets: results, challenges, and questions}
We start this part with a truth derived from Proposition 3.2 (in semigroups and groups) that $BA\subseteq A$ if and if
and only if $A=\langle B\rangle^1 E$ for some $E\subseteq S$. Hence, due to Proposition 3.2 (also see Remark 2.20),
 we have the following useful general forms for left upper $B$-periodic sets where $E,D$ are arbitrary subsets:\\
(a) $A=\langle B\rangle^1 E$.\\
(b) $A=\langle B\rangle D\cup \langle B\rangle^1 E$.\\
(c) $A=\langle \langle B\rangle\rangle D\cup \langle B\rangle^1 E$ if $\langle \langle B\rangle\rangle$ exists
(especially in groups).\\
Of course, these general forms do not necessarily give us the uniqueness or directness
 of the components and their products (e.g., see Observation 1-4 of Subsection 3.3). But, Lemma 2.2 and Theorem 2.11,
 provide the uniqueness and the directness of the representation, under some conditions,
 and moreover lead us to the following complete characterization of all left upper periodic sets
by two explicit formulas.
\begin{theorem}[Characterization of left upper periodic sets]
Under the assumptions of Theorem 2.11, if $B$ has the left summand property, then all left upper $B$-periodic sets  $A$
are obtained from the following forms:
\begin{equation}
A=\mathcal{B} D\dot{\cup} B^1(\Sigma^\ell_{Y|B}\setminus Y) \;\;\; ; \;\;\; D\subseteq \mathcal{D},\; Y\subseteq S
\end{equation}
\begin{equation}
A=\mathcal{B} D\dot{\cup} B^1(E_B(\Sigma^\ell_{Y|B})) \;\;\; ; \;\;\; D\subseteq \mathcal{D},\; Y\subseteq S
\end{equation}
$($One can conclude that the products in the both formula are direct.$)$
\end{theorem}
\begin{proof}
This is a direct result of Theorem 2.11 and Theorem 4.8 of \cite{ulpss} (also, see ``characterization of anti-left $B$-transference subsets by the starts'' in Section 2).
\end{proof}
Before examining the uniqueness and/or directness conditions for the mentioned general forms (especially (c)), we need
some results as follows.
An interesting fact that we found out during the research was that
in most of the situations that we place for the uniqueness of the representation $A=\mathbb{B}D\dot{\cup} B^1E$,
$B$ must be a sub-semigroup. Therefore, a necessary and sufficient condition for
 it is obtained in the following theorem.
\begin{theorem}
Let $D$ be an arbitrary subset of $S$ and $\emptyset\neq B\subseteq\mathbb{B}\leq_\ell S$. Then, a
necessary and sufficient condition for $A=\mathbb{B}D\dot{\cup} B^1E$ with $E\neq\emptyset$,
$E\cap \mathbb{B}^*E=\emptyset$ to be left upper $B$-periodic is $B\dot{\leq} S$.
\end{theorem}
\begin{proof}
It is obvious that if $B\dot{\leq} S$, then $BA\subseteq A$
(since $BB^1=B$).
Now, let $BA\subseteq A$ and $b\in B$. Then
$$
bA=\mathbb{B}D\cup bB^1E\subseteq \mathbb{B}D\dot{\cup} B^1E=A
$$
We claim that $\mathbb{B}D\cap bB^1E=\emptyset$. Because
$$
x\in bB^1E\cap \mathbb{B}D\Leftrightarrow b^{-1}x\in b^{-1}(bB^1E\cap \mathbb{B}D)
\subseteq (b^{-1}b)B^1E\cap (b^{-1}\mathbb{B})D=B^1E\cap \mathbb{B}D=\emptyset
$$
Hence, $bBE\subseteq bB^1E\subseteq B^1E=BE\cup E$ and so $bBE\subseteq BE$.
For if $bb'e_1=e_2$ for some $b'\in B$ and $e_1,e_2\in E$, then
$\mathcal{B}^*E\cap E=\emptyset$ implies $bb'=1_\mathcal{B}$
which implies $B\cap B^{-1}\neq\emptyset$ and so $BA=A$ (by Remark 3.3(b)) that is
a contradiction.\\
 Therefore, $bBE\subseteq BE$
 which means for every $b\in B$ and $b_1\in B$ there exist
$b'_1\in B$ and $e'\in E$ such that $bb_1e=b'_1e'$ and so
$e=(bb_1)^{-1}b'_1e'\in E\cap \mathbb{B}E$. Thus $E\cap \mathbb{B}^*E=\emptyset$
requires that $(bb_1)^{-1}b'_1=1_\mathcal{B}$ and so $bb_1=b'_1\in B$. This completes the proof.
\end{proof}
The above result agrees with
Proposition 3.2(a) which assures us that in the study of upper $B$-periodic sets, we can assume
that $B$ is a sub-semigroup of $S$, without loss of the generality.
It is worth noting that under the assumption $B\dot{\leq} S$ we have
$A$ is left $B$-periodic if and only if $A=B^1E$ for some $E\subseteq S$.
\\
Note that in the above theorem if $\mathbb{B}=\mathcal{B}$ (mentioned in Remark 1.4),
then we can reduce the condition $E\cap \mathbb{B}^*E=\emptyset$ to $E\cap BE=\emptyset$
(by Lemma 2.8). But always we can not do it, for example, let $S$ be the additive group of real numbers,
$B:=\{1\}$, $\mathbb{B}:=\mathbb{Z}$, $D=\emptyset$, and $E:=\mathbb{N}_e$ (the set of even natural numbers).
Then, $A=\{0,1\}+\mathbb{N}_e=\{2,3,4,\cdots\}$, $1+A\subseteq A$ and $E\cap (B+E)=\emptyset$ but
$E\cap (\mathbb{B}^*+E)=E\neq \emptyset$ and $B$ is not a sub-semigroup. Not that, in this example,
the summation $B+E$ (and $B^1+E$) is direct.
\begin{cor}
Let $\mathbb{B}\leq_\ell S$, $\mathbb{B}_2=\{1\}$, and $A=\mathbb{B}_+^1\cdot E$ for some
subset $E\neq\emptyset$. Then
$$
\mathbb{B}_+A\subseteq A \Leftrightarrow \mathbb{B}_+\dot{\leq} \mathbb{B}
$$
Moreover, if one of the above equivalent conditions holds, then $\Pk^\ell_\mathbb{B}(A)=\emptyset$.
\end{cor}
\begin{proof}
First note that $E$ is anti-left $\mathbb{B}^*$-transference, by Lemma 2.8. Then, the above theorem proves the
first part. The second part ($\Pk^\ell_\mathbb{B}(A)=\emptyset$) is concluded from Lemma 2.16.
\end{proof}
\subsection{Representation of upper periodic subsets under weaker conditions}
Due to the basic properties, the main theorem, and the general forms there are three main subjects for upper periodic subsets
as: representations, uniqueness of a representations, directness of the representations,
and of course extra properties of their components.
To achieve their basic features it is better to
consider the $\mathbb{B}$-periodic kernel of $A$, where $\mathbb{B}$ is also an arbitrary subset, and then we
add some conditions on $\mathbb{B}$ (e.g., $\mathbb{B}\dot{\leq} X$, $B\mathbb{B}\subseteq \mathbb{B}$, $\mathbb{B}=B$, etc.).\\
If $BA\subseteq A$, then
 $$
BA=BC_\mathbb{B}\cup BF_\mathbb{B}\subseteq C_\mathbb{B}\dot{\cup} F_\mathbb{B}
$$
Now, let $X=S$ be a semigroup, \underline{$B\mathbb{B}\subseteq \mathbb{B}$} (i.e., $\mathbb{B}$ is left upper $B$-periodic), and
\underline{$B^{-1}\subseteq \mathbb{B}\cup B$}.
Then, putting $C:=C_\mathbb{B}(A)$, $F:=F_\mathbb{B}(A)$, and $E=E_B(A)$ we have
$BC=B\mathbb{B}C\subseteq \mathbb{B}C=C$, and $BF\cap C=\emptyset$. For if $BF\cap C\neq \emptyset$, then
$bf\in C$ for some $b\in B$, $f\in F$ and so
$$
f\in b^{-1}C\subseteq B^{-1}C\subseteq (\mathbb{B}\cup B)C=\mathbb{B}C\cup BC=C,
$$
that is a contradiction. Therefore, $C$ and $F$ both are left upper $B$-periodic
and so
$$BC_\mathbb{B}\cap BF_\mathbb{B}\subseteq C_\mathbb{B}\cap F_\mathbb{B}=\emptyset\; , \; BA=BC_\mathbb{B}\dot{\cup} BF_\mathbb{B}$$
Now, suppose that \underline{$B^{-1}\cap\mathbb{B}\neq \emptyset $}
(e.g., if $B^{-1}\nsubseteq B$ when the above mentioned assumption $B^{-1}\subseteq \mathbb{B}\cup B$ holds), then
we have $E_B\subseteq F_\mathbb{B}$. For if $x\notin BA$, then $b^{-1}x\notin A$ for all $b\in B$ (i.e., $B^{-1}E\cap A=\emptyset$)
thus $\mathbb{B}x \nsubseteq A$ and so $x\notin C$. Therefore,
$$
B^1E=BE\cup E\subseteq BF\cup E\subseteq F\cup E=F
$$
which means $B^1E_B\subseteq F_\mathbb{B}$. Now, we claim that
$$
F\setminus E=F\cap BA=BA\setminus C=BF.
$$
For if $f\in F\cap BA$, then $f=b_0a_0$ for some $b_0\in B$ and $a_0\in A$, thus $a_0=b_0^{-1}f\in F$
(because $b_0^{-1}f\in C$ implies $f\in b_0C\subseteq C$). This shows that $F\cap BA\subset BF$, and proof of other
items are easy by using the above inclusion properties. Also $E\cap BE\subseteq E\cap BF=E\cap (F\setminus E)=\emptyset$ which means
$E$ is anti-left $B$-transference. Up to now we prove that if
the above underlying assumptions hold (e.g., if $B\subseteq \mathbb{B}\leq_\ell S$, and of course $BA\subseteq A$),
then $A=C_\mathbb{B}\dot{\cup} (BF_\mathbb{B}\dot{\cup} E_B)$. Hence, under the assumptions, if one of the following equivalent conditions
(that we say $A$ is well left $(\mathbb{B},B)$-started) holds:
\begin{equation*}
BF_\mathbb{B}\subseteq BE_B \Leftrightarrow F_\mathbb{B}=B^1E_B
\Leftrightarrow BF_\mathbb{B}=BE_B \Leftrightarrow BA\setminus C_\mathbb{B}\subseteq BE_B\Leftrightarrow BF_\mathbb{B}\cup E_B=
B^1E_B=F_\mathbb{B}
\end{equation*}
then  $A=C_\mathbb{B}\dot{\cup} B^1E_B=\mathbb{B}C_\mathbb{B}\dot{\cup} B^1E_B$. Therefore, if
the underlying assumptions hold and $A$ is well left $(\mathbb{B},B)$-started
set, then $A$ enjoys the representation $A=\mathbb{B}D\dot{\cup} B^1E$ (not necessarily unique or direct), and so we arrive at the following theorems
and corollaries (and then essential projects, ideas, and questions).
\begin{theorem}
Let $S$ be a semigroup with a $($fixed$)$ left identity $l$ and $B^{-1}=B_\ell^{-1}(l)$ exists.
Also, let $\mathbb{B}$ be a subset of $S$
such that
\begin{equation}
B\mathbb{B}\subseteq \mathbb{B}\; , \; B^{-1}\subseteq \mathbb{B}\cup B\; , \; B^{-1}\cap\mathbb{B}\neq \emptyset
\end{equation}
Then,  every left upper $B$-periodic subset $A$ satisfies the identity\\
$($a$)$ $$A=C_\mathbb{B}\dot{\cup} (BF_\mathbb{B}\dot{\cup} E_B)$$
$($b$)$ Hence, there exist left upper $B$-periodic subsets $C,F$, and an anti-left $B$-transference set $E$  such that
$C$ is left $\mathbb{B}$-periodic and $A=C\dot{\cup} (BF\dot{\cup} E)$.
Moreover we have
$$BF=BA\cap F=F\setminus E\; , \; A=\mathbb{B}C\dot{\cup} (BF\dot{\cup} E)$$
\end{theorem}
\begin{cor}
Under the conditions of Theorem 3.8, we have
\begin{equation}
BF_\mathbb{B}\subseteq BE_B \Leftrightarrow F_\mathbb{B}=B^1E_B
\Leftrightarrow BF_\mathbb{B}=BE_B \Leftrightarrow BA\setminus C_\mathbb{B}\subseteq B^1E_B
\end{equation}
and if one of the equivalent conditions $(3.4)$ holds, then
\begin{equation}
A=C_\mathbb{B}\dot{\cup} B^1E_B=\mathbb{B}C_\mathbb{B}\dot{\cup} B^1E_B
\end{equation}
Therefore, there exist $D,E\subseteq S$ such that
\begin{equation}
A=\mathbb{B}D\dot{\cup} B^1E\; , \; E\cap BE=\emptyset
\end{equation}
\end{cor}
The above theorem and corollary lead us to another theorem, although its condition is
 much weaker than the main theorem, but its ruling is significant. In that case, we have
  reached a similar representation, and in order to approach its directness and uniqueness,  other
  conditions can be added to it in several steps.
\begin{theorem}
If $\emptyset\neq B\subseteq \mathbb{B}\leq_\ell S$
$($or with the weaker condition: $\emptyset\neq B\cup B^{-1}\subseteq \mathbb{B}\dot{\leq}_\ell S)$, then every
well left $(\mathbb{B},B)$-started upper $B$-periodic set $A$ has the form $(3.6)$.
\end{theorem}
\begin{ex}
Consider the two associative binary operations ``$\cdot$'' (ordinary multiplication) and
``$\bar{\cdot}$'' in Example 1.2 for real numbers and put $Y^\circ:=Y\setminus\{0\}$ for every $Y\subseteq \mathbb{R}$
(this is different to $\mathbb{R}^\ast:=\mathbb{R}\setminus\{1\}$ when $Y=\mathbb{R}$, since $0$ is not an identity).
Note that $\mathbb{R}^\circ$ is a sub-semigroup of $\mathbb{R}$ relative to the both operations, and
$x\bar{\cdot}y=xy$ whenever $x>0$.
Fix a subset $\emptyset\neq B\subseteq (0,1]$ and put $\beta_0=\inf B$ and $\beta_1=\sup B$. Then,
for every $M>0$, the subsets $(-\infty,M)$ and $A:=(-\infty,M)^\circ$ are left upper $B$-periodic.
It is easy to show that $B\bar{\cdot} A=BA=\beta_1A$ and so
$$
E_B(A)=(-\infty,M)^\circ\setminus (-\infty,\beta_1M)^\circ=(0,M)\setminus (0,\beta_1M)=
\left\{\begin{array}{l}
 [\beta_1M,M)  \quad : \quad  \beta_1\neq 1 \\
\emptyset  \qquad \qquad\;\;\; : \quad \beta_1=1
\end{array} \right.
$$

Thus, $A$ is left $B$ periodic if and only if $\beta_1=1$, and if $\beta_1\neq 1$, then
\begin{equation}
BE_B=\bigcup_{b\in B}[b\beta_1M,bM)
\end{equation}

But the above union $(3.7)$ had different values respect to
 the different amounts that the supremum or infimum of $B$ take, and therefore various examples are produced.
As some important cases, we have:\\
- if $\beta_1^2\leq \beta_0\in B$, then $BE_B=[\beta_0\beta_1M,\beta_0M)$ and $B^1E_B=[\beta_0\beta_1M,M)$.\\
- if $B$ contains a sequence $b_n$ such that
$$
\bigcup_{n=1}^\infty[b_n,\frac{1}{\beta_1}b_n)=(0,1)
$$
(e.g., if $\beta_1=\frac{1}{n_0}$, for some integer $n_0>1$, and $b_n=\frac{1}{n+n_0-1}$),
then $BE_B=(0,\beta_1M)$ and $B^1E_B=(0,M)$.\\
Now, for checking the conditions of Theorem 3.9, as an example, if
$\mathbb{B}:=(0,+\infty)\leq_\ell \mathbb{R}^\circ$ and
$B:=(0,\frac{1}{2}]\dot{\leq} \mathbb{B}$, then $BE_B=(0,\frac{1}{2}M)$ and
$$
BA\cap \mathbb{B}A^c=(-\infty,\beta_1M)^\circ\cap (0,+\infty)=(0,\frac{1}{2}M)\subseteq BE_B
$$
and hence Theorem 3.9 implies that
$(-\infty,M)^\circ=A=(0,+\infty) \{-1\}\dot{\bigcup} (0,\frac{1}{2}]^1 [\frac{M}{2},M)$.\\
Since this example does not satisfy all conditions of the main theorem ($B$ is not a positive subset
of $\mathbb{B}$) it is logical to check whether this representation is direct and/or unique.
Regarding this explanation, the product $(0,+\infty) \{-1\}$ is direct, as we expect by Lemma 2.8,
but the product $(0,\frac{1}{2}]^1 [\frac{M}{2},M)$ is not direct since $$(0,\frac{1}{2}]^{-1}(0,\frac{1}{2}]\cap [\frac{M}{2},M)
[\frac{M}{2},M)^{-1}=[\frac{1}{2},+\infty)(0,\frac{1}{2}]\cap [\frac{M}{2},M)(\frac{1}{M},\frac{2}{M}]=(\frac{1}{2},2)\neq \{1\}$$
Hence, this shows the importance of Theorem 3.9 since Theorem 2.11 is not applicable to this case.
\end{ex}
\begin{ex}
Let $G$ be the additive group  $(\mathbb{R},+)$, $B=\mathbb{Q}_+$ and $\mathbb{B}=\mathbb{Q}$.
Then $A=[M,+\infty)$ is upper $B$-periodic (but not $B$-periodic) and $\emptyset\neq B\subseteq \mathbb{B}\leq G$.
Therefore, we have $(3.3)$ and so the property (a) of Theorem 3.7 is satisfied. It is worth noting that,
 here, $\Pk_\mathbb{B}(A)=\emptyset$, $F_\mathbb{B}=A$
(i.e., $A$ is $\mathbb{Q}$-periodic free),  $B+A=(M,+\infty)$, and
$\St_B(A)=A\setminus (B+A)=\{M\}$. Thus $A$ is not well $(\mathbb{B},B)$-started,
since $(B+A)\cap (\mathbb{B}+A^c)=B+A=(0,+\infty)$, $B^0+(A\setminus B+A)=\mathbb{Q}_+^0+M$)
and so the last condition of Theorem 3.9 is not satisfied.
One can check that $(3.6)$ does not hold.
\end{ex}
\subsection{Uniqueness and directness of the representation}
Now, we try to obtain the uniqueness and/or directness of the representation $(3.6)$ by adding some extra conditions ( but weaker than
the conditions of Theorem 2.11).\\
Let $B,\mathbb{B}$ be fixed non-empty subsets of $S$ such that \underline{$B\dot{\leq} S$} and \underline{$B\mathbb{B}=\mathbb{B}$},
and let $D,E$ be arbitrary subsets of $S$. Then\\
{\bf Observation 1.} The subset $A:=\mathbb{B}D\cup B^1E$ is left upper $B$-periodic and $BA=\mathbb{B}D\cup BE$
(thus $A=BA\cup E$ and $\mathbb{B}D$ is left $B$-periodic). Moreover, $E=A\setminus BA$ if and only if \underline{$E\cap BE=E\cap \mathbb{B}D=\emptyset$}
(equivalently,  $A=BA\dot{\cup} E$). \\
Note that $E\cap BE=E\cap \mathbb{B}D=BE\cap \mathbb{B}D=\emptyset$ if and only if
$A=\mathbb{B}D\dot{\cup} B^1E$ and $E\cap BE=\emptyset$, and if and only if $A=\mathbb{B}D\dot{\cup} BE\dot{\cup} E$ . \\
{\bf Observation 2.} Therefore, if \underline{$A=\mathbb{B}D\dot{\cup} B^1E$} and \underline{$E$ is anti-left $B$-transference},
then $E=E_B(A)=\St^\ell_B(A)$. \\
{\bf Observation 3.} From Observation 2 we conclude that if $A$ is a given left upper $B$-periodic subset of the form
$A=\mathbb{B}D\dot{\cup} B^1E$
with $E\cap BE=\emptyset$, then $E$ and  $\mathbb{B}D$ are unique, which means
$$
A=\mathbb{B}D_1\dot{\cup} B^1E_1=\mathbb{B}D_2\dot{\cup} B^1E_2\Rightarrow E_1=E_2 , \mathbb{B}D_1=\mathbb{B}D_2,
$$
(when $E_1,E_2$ are anti-left $B$-transference).\\
{\bf Observation 4.} For directness and completing the uniqueness of the representation, note that
if $\mathbb{B}\leq_\ell S$, then $\mathbb{B}D=\mathbb{B}\cdot D$ (resp. $B^1E=B^1\cdot E$) is equivalent to $D\cap \mathbb{B}^*D=\emptyset$
(resp. $E\cap \mathbb{B}^*E=\emptyset$), by Lemma 1.3. Therefore, under the initial assumption  $A=\mathbb{B}D\dot{\cup} B^1E$:  \\
{\bf (a)} If $B\dot{\leq} S$, $B\mathbb{B}=\mathbb{B}$,
and $D\cap \mathbb{B}^*D=E\cap \mathbb{B}^*E=\emptyset$ then all items of the unique direct representation $A=\mathbb{B}\cdot D\dot{\cup} B^1\cdot E$
hold except probably uniqueness of $D$.\\
{\bf (b)}  If $\mathbb{B}^*\setminus B^{-1}\subseteq B\dot{\leq} \mathbb{B}\leq_\ell S$, then ($B\dot{\leq} S$, $B\mathbb{B}=\mathbb{B}$ and) $E\cap \mathbb{B}^*E=\emptyset$ is equivalent to
$E\cap BE=\emptyset$ (by Lemma 2.8) and so we obtain both uniqueness and directness of the product $B^1\cdot E$ if
$E$ is anti-left $B$-transference (by Observation 2).\\
{\bf (c)} If $B\dot{\leq} \mathbb{B}\leq_{\ell f} S$ and $D\subseteq \mathcal{D}$, where $\mathcal{D}$ is the fixed right transversal of $\mathbb{B}$ in $S$,
 then we obtain both uniqueness and directness of the product $\mathbb{B}\cdot D$ by Lemma 2.2.\\
{\bf (d)} Hence, from (b) and (c) we conclude that if $\mathbb{B}^*\setminus B^{-1}\subseteq B\dot{\leq} \mathbb{B}\leq_{\ell f} S$, $E\cap BE=\emptyset$,
and $D\subseteq \mathcal{D}$ then the (full) unique direct representation $A=\mathbb{B}\cdot D\dot{\cup} B^1\cdot E$ is satisfied.\\

Combining the results of Subsection 3.2 and this subsection, one can obtain the representation of left upper periodic
sets with various properties regarding the directness and the uniqueness. For example, we arrive at the following theorem.
\begin{theorem} $($Some different representations of left upper periodic subsets of semigroups and groups$)$. \\
$($a$)$ If $B\dot{\leq} \mathbb{B}\leq_\ell S$, then every well left $(\mathbb{B},B)$-started upper $B$-periodic set $A$
has the $($second-half$)$ unique representation $A=\mathbb{B}D\dot{\cup} B^1 E$ relative to $E$ $(E$ is unique and anti-left $B$-transference$)$.
Hence its periodic free component has the unique representation $F_B(A)=B^1E$ $($but not necessarily direct$)$.\\
$($b$)$ If $\mathbb{B}^*\setminus B^{-1}\subseteq B\dot{\leq} \mathbb{B}\leq_\ell S$, then every well left $(\mathbb{B},B)$-started upper $B$-periodic set $A$
has the $($second-half$)$ unique and direct representation $A=\mathbb{B}D\dot{\cup} B^1\cdot E$ relative to $E$ $(E$ is unique and anti-left $B$-transference$)$. Hence its periodic free component has the unique direct representation $F_B(A)=B^1\cdot E$.\\
$($c$)$ If $B\dot{\leq} \mathbb{B}\leq_{\ell f} S$, then every well left $(\mathbb{B},B)$-started upper $B$-periodic set $A$ has the $($first-half$)$
unique and direct representation $A=\mathbb{B}\cdot D\dot{\cup} B^1 E$ with $D\subseteq \mathcal{D}$ $($for a fixed right transversal $\mathcal{D}$ of $\mathbb{B}$ in $S)$ respect to $D$.\\
$($d$)$
$($Another version of the main Theorem 2.11$)$ If $\mathbb{B}^*\setminus B^{-1}\subseteq B\dot{\leq} \mathbb{B}\leq_{\ell f} S$, then every well left $(\mathbb{B},B)$-started upper $B$-periodic set $A$ has the $($full$)$ unique and direct representation $A=\mathbb{B}\cdot D\dot{\cup} B^1\cdot E$ with $D\subseteq \mathcal{D}$.
\end{theorem}
\begin{proof}
Since the conditions of each of the four parts require $\emptyset\neq B\subseteq \mathbb{B}\leq_\ell S$, if
$A$ is a well left $(\mathbb{B},B)$-started upper $B$-periodic set then $A=\mathbb{B}D\dot{\cup} B^1E$ and $E\cap BE=\emptyset$,
by Theorem 3.9. Then, Observation 4 completes the proof.
\end{proof}
\begin{ex}
Let $S=\mathbb{R}$ be the multiplicative semigroup of real numbers, $M>0$, $A:=(0,M)$, and $\emptyset\neq B\subseteq (0,1]$.
Then, $\mathbb{B}=(0,+\infty)$ is a factor subgroup of $S$, $B\dot{\leq} \mathbb{B}\leq_\ell S$, $A$ is $\mathbb{B}$-periodic free, and $BA\subseteq A$.
Putting $\beta_0:=\inf B$, and $\beta_1:=\sup B$, similar to Example 3.10 we have
$BA=\beta_1 A$,
$$
E_B(A)=A\setminus BA=(0,M)\setminus (0,\beta_1M)=
\left\{\begin{array}{l}
 [\beta_1M,M)  \quad : \quad  \beta_1\neq 1 \\
\emptyset  \qquad \qquad\;\;\; : \quad \beta_1=1
\end{array} \right.
$$
and
\begin{equation}
[\beta_1^2M,\beta_1M)\subseteq BE_B(A)=\bigcup_{b\in B}[b\beta_1M,bM)=M[\beta_1,1)B\subseteq [\beta_0\beta_1M,\beta_1M)
\end{equation}
(with the convention that $[a,b)=\emptyset$ if $a=b$).\\
Therefore $A$ is well $(\mathbb{B},B)$-started if and only if $(0,\beta_1M)\subseteq \bigcup_{b\in B}[b\beta_1M,bM)$. Hence:\\
- If $\beta_0>0$ or $\beta_1=1$, then $A$ is not well $(\mathbb{B},B)$-started.\\
- If $\beta_0=0$ and $\beta_1\neq 1$, then $A$ may be well $(\mathbb{B},B)$-started or not (once take $B:=(0,\frac{1}{2}]$ and another one
consider $B:=\{\frac{1}{10}, \frac{1}{11}, \frac{1}{12},\cdots\}\cup \{\frac{9}{10}\}$).\\
Hence, for this example the conditions of Theorem 3.12(b),(c) never are satisfied, but for Theorem 3.12 (a) we can obtain so many
cases. For instance, if $\beta_1\neq 1$ and $B$ is dense in $[0,\beta_1]$ (e.g., if $\beta_0=0$, $\beta_1\neq 1$, and $B\dot{\leq}(0,\beta_1]$) then $A$ is well $(\mathbb{B},B)$-started, and all the conditions for Theorem 3.12 (a) hold, and so we have the unique representation $A=B^1E$ (but not necessarily direct).
\end{ex}
{\bf Problem and Question VI.} It is a good idea for some readers to answer the following regarding Example 3.13 and Theorem 3.12:\\
- $B(A\setminus BA)$ takes its most possible case (i.e., $[\beta_0\beta_1M,\beta_1M)$) if $\beta_1^2\leq \beta_0\in B$.
 Is the converse also true?\\
- Is it true that $B(A\setminus BA)=(\beta_0\beta_1M,M)$) if $\beta_1^2\leq \beta_0\notin B$?\\
- What happens if $\beta_0<\beta_1^2$?\\
- When does $B(A\setminus BA)$ take its least possible case (i.e., $[\beta_1^2M,\beta_1M)$)?\\
Also, give a bunch of examples of subsets $A$ satisfying the conditions of Theorem 3.12 (c) (resp. Theorem 3.12 (b))
but not Theorem 3.12 (d) (resp. Theorem 3.12 (c)).\\
\textbf{Projections onto factor subsets.}
One of the important tools for a deep study of the topic, especially for sub-structures, is
the concept of projections induced by direct products of subsets and direct representations of periodic type subsets.\\
 Every factorization of a subset $Y$ of a magma by two subsets induces
two (unique) projections onto the factors as follows.
If $Y=\Delta\cdot\Omega$, then
we have the surjective maps $P_{\Delta
}:Y\longrightarrow \Delta$, $P_{\Omega} :Y\longrightarrow \Omega$
defined by $P_{\Delta} (x)=\delta$, $P_{\Omega}(x)=\omega$, for every $x\in
Y$ with $x=\delta \omega$, where $\delta\in \Delta$ and $\omega\in \Omega$. The maps $P_{\Delta}$, $P_{\Omega}$ are
called left and right projections with respect to the factorization $Y=\Delta\cdot\Omega$.
 Now, we call a function $f:X\longrightarrow X$ left (resp. right)
projection, if there exists a left (resp. right) factor $\Delta$ (resp. $\Omega$)
of $X$ such that $f=P_{\Delta}$ (resp. $f=P_{\Omega}$).\\
Now, let $BA\subseteq A$, and $A$ satisfies some of the conditions of Theorem 3.12(b)-(d). Then, the
first-half (resp. second-half) directness  of the representation $A=\mathbb{B}\cdot D\dot{\cup} B^1E$ (resp. $A=\mathbb{B}D\dot{\cup} B^1\cdot E$)
provides the  projections $P_D:\mathbb{B}\cdot D\rightarrow D$ (resp. $P_E:B^1\cdot E\rightarrow E$). Hence, if $A$ has the full direct representation
 $A=\mathbb{B}\cdot D\dot{\cup} B^1\cdot E$, then one can define a (surjective) bi-projection
$P_{D,E}:A\rightarrow D\dot{\cup} E$ by
\begin{equation*}
 P_{D,E}(x):=\left\{\begin{array}{l}
P_D(x) \quad : \quad x\in \mathbb{B}\cdot D\\
P_E(x) \quad : \quad x\in B^1\cdot E
\end{array} \right.
\end{equation*}
We will use this bi-projection in future research and studies.
As an interesting special case, for every upper 1-periodic real set $A$
we have  $A=(\mathbb{Z}\dot{+}D)\dot{\cup}(\mathbb{Z}_+^0\dot{+}E)$,
and so (due to Lemma 4.2)
\begin{equation*}
 P_{D,E}(x):=\left\{\begin{array}{l}
x+\lceil-x\rceil \qquad\quad\quad  \qquad :\; x\in \Pk(A)\\
x+\min\big((A-x)\cap \mathbb{Z}\big) \; : \; x\in \Pf(A)
\end{array} \right.
\end{equation*}
Note that $P_{D,E}(x+1)=P_{D,E}(x)$, for all $x\in A$.
\begin{ex}
Consider the additive group of real numbers, fix $b\in \mathbb{R}\setminus\{0\}$, and put
$\mathbb{R}_b:=b[0,1)$ (that is $[0,b)$ or $(b,0]$).
The $b$-integer part function $\lfloor x\rfloor_b:=b\lfloor\frac{x}{b}\rfloor$
and $b$-decimal ($b$-fractional) part function $\{ x\}_b=b\{\frac{x}{b}\}$ (where
$\{x\}:=x-\lfloor x\rfloor$, see \cite{grplike}) are the two projections
with respect to the (additive) factorization
$\mathbb{R}=\langle\langle b\rangle\rangle\dot{+}\mathbb{R}_b$
(note that $\langle\langle b\rangle\rangle=b\mathbb{Z}$). In particular, putting $b=1$, we have
$P_{\mathbb{Z}}=\lfloor\; \rfloor$ and $P_{[0,1)}=\{ \; \}$.
\end{ex}
\section{Periodic types real sets, real groups, and semigroups with related classification and characterizations}
Given that the initial idea of this broad topic was the set of real numbers,
 and as we expect, there are more interesting properties for sets of
 periodic type real sets that we discuss here. Therefore, this part may attract the
  attention of experts in the field of real numbers.
\subsection{Periodic types real sets and concenterable upper periodic cases}
First, notice that all the examples that have been given about the set of numbers,
in this article (so far) are also related to this section. Then, by applying general
  theorems for sets of numbers (real, rational, integer, complex, etc.), several special
 interesting results can be obtained some of which are given below.
\begin{lem}[Two characterizations for positive and negative integers]
Let $B$ be a subset of real numbers.\\
$($a$)$ If $-(\mathbb{Z}^*\setminus B)\subseteq B$ and $B$ is closed under the addition,
then $B$ is either $\mathbb{Z}_+$ or $\mathbb{Z}_-$ $($and vice versa$)$.\\
$($b$)$ If $-(\mathbb{Z}^*\setminus B)\subseteq B$ and there is $A\subseteq \mathbb{R}$
such that
$$
B+A\subseteq A\; , \; (B+A)\cap (\mathbb{Z}+A^c)\subseteq B^0+(A\setminus(B+A)),
$$
then $B$ is either $\mathbb{Z}_+$ or $\mathbb{Z}_-$ $($and vice versa$)$.
\end{lem}
Let $b\neq 0$ be a real number and $A$ be a subset of the additive group of real numbers.
Since $b+A\subseteq A$ if and only if $1+(\frac{1}{b}A)\subseteq (\frac{1}{b}A)$,
and equivalently $A\subseteq (-b)+ A$, we can study
upper 1-periodic subsets instead all upper and lower $B$-periodic subsets where $B$ is a singleton (without loss of generality). \\
Hence, in continuation, for every real set $A$, we put $\Pk(A):=\Pk_\mathbb{Z}(A)=\Pk_{\{1\}}(A)$,
$\Pf(A)=A\setminus\Pk(A)$, and $\St(A):=\St_{\{1\}}(A)=\St_1(A)$.
\begin{lem}
For every upper 1-periodic set $A$ we have\\
$($a$)$
$$
\Pk(A)=\{x\in \mathbb{R}: (A-x)\cap \mathbb{Z}=\mathbb{Z}\}=\{x\in A: (A-x)\cap \mathbb{Z}=\mathbb{Z}\}
$$
$($b$)$
$$
\Pf(A)=\{x\in A: (A-x)\cap \mathbb{Z} \mbox{ is bounded below} \}
$$
$($c$)$
$$
\St(A)=\{x\in A: \min\big((A-x)\cap \mathbb{Z}\big)=0\}=\{x\in \Pf(A): \min\big((A-x)\cap \mathbb{Z}\big)=0\}=\St(\Pf(A))
$$
$($d$)$ Regarding the factorization $\Pf(A)=\mathbb{Z}_+^0\dot{+}\St(A)$ we have the followings for its projections:
$$
P_{\St(A)}(x)=x+\min\big((A-x)\cap \mathbb{Z}\big)\; , \; P_{\mathbb{Z}_+^0}(x)=\max\big((A-x)\cap \mathbb{Z}\big)\; ; \; x\in \Pf(A)
$$
\end{lem}
\begin{proof}
For every $x\in \mathbb{R}$ we have
$$
(A-x)\cap \mathbb{Z}=\mathbb{Z}\Leftrightarrow \mathbb{Z}\subseteq A-x\Leftrightarrow \mathbb{Z}+x\subseteq A\Leftrightarrow
x\in \Sigma_{A|\mathbb{Z}}=\Pk(A).
$$
Also, $0\in \mathbb{Z}= (A-x)\cap \mathbb{Z}$ implies  $x\in A$. Hence, we are done (a).
From (a) we conclude that $x\in \Pf(A)$ if and only if there exists $k_0\in \mathbb{Z}$ such that
$k_0\neq a-x$ for all $a\in A$. We claim that $k_0$ is a lower bound for $(A-x)\cap \mathbb{Z}$.
For if $k<k_0$ is an integer with $k\in (A-x)\cap \mathbb{Z}$, then
$$x+k_0\in x+k+\mathbb{Z}_+^0 \subseteq A+\mathbb{Z}_+^0\subseteq A,$$
that is a contradiction.\\
For (c) and (d), put $m_x:=\min\big((A-x)\cap \mathbb{Z}\big)$. If $m_x=0$, then $0\in A-x$ and so $x\in A$.
But $x-1\in A$ requires $-1\in (A-x)\cap \mathbb{Z}$ which contradicts $m_x=0$. Conversely, if $x\in \St(A)$,
then $m_x\leq 0$ since $0\in A-x$. Now, if $m_x\leq -1$, then
$$
x-1\in x+m_x+\mathbb{Z}_+^0\subseteq A+\mathbb{Z}_+^0\subseteq A,
$$
that is a contradiction. Thus we are done (c). Finally,
the identity
$$
(A-x)\cap \mathbb{Z}=\big((\Pk(A)-x)\cap \mathbb{Z}\big) \cup \big((\Pf(A)-x)\cap \mathbb{Z}\big)
$$
implies $(A-x)\cap \mathbb{Z}=(\Pf(A)-x)\cap \mathbb{Z}$ for every $x\in \Pf(A)$ (for if
$y-x=k$ for some $y\in \Pk(A)$ and $k\in \mathbb{Z}$, then $x=y-k\in \Pk(A)$ a contradiction).
Now if $x\in \Pf(A)$, then $x=e+k$ for some $e\in \St(A)$ and $k\in \mathbb{Z}$. Then,
$$
(\Pf(A)-x)\cap \mathbb{Z}=((e+\mathbb{Z}_+^0)-(e+k))\cap \mathbb{Z}=\mathbb{Z}_+^0-k
$$
This not only proves (d) but also shows that
$$
(\Pf(A)-x)\cap \mathbb{Z}=\{P_{\mathbb{Z}_+^0}(x), P_{\mathbb{Z}_+^0}(x)+1,\cdots\}
$$
\end{proof}
Another interesting and applicable concept that we have achieved, which is specific to subsets
of real numbers, is the following new definition. Due to the interpretation of the start of an upper periodic set,
we can consider a nice geometric description of it in the real line.
Indeed, $x\in \St(A)$ if and only if all the points on the right side of it whose distance is an integer number belong
 to $A$, but no point on the left side of it with this property belongs to $A$.
 In other words, its positive integer shift belongs to A and the negative integer shift does not belong to it.
 For this reason, we call the points of $\St(A)$  the starting points and $\St(A)$ the start of $A$.
The starting points
can generates all non-periodic points of $A$ (i.e., elements of $\Pf(A)$)
with non-negative integer shifts, and others (periodic points of $A$, i.e., elements of $\Pk^\ell_1(A)$)
can be generate by the unique subset $D$ of decimal (fractional) numbers with integer shifts.
Also, we observe that the start points of non-periodic points are the same as $\St(A)$, i.e., $\St(\Pf(A))=\St(A)$.
This is a clear and geometric interpretation of the real case for the facts
that were stated and proved in the Theorem 2.11 and Remark 2.13 (in general). 
But one of the useful and ideal cases is that the non-periodic generator points belong to
 an interval of length one, like what we were able to obtain for periodic generators ($D\subseteq [0,1)$).
 This motivation brings us to the topic of concenterable upper periodic real sets and sub-semigroups in the line of real numbers.
\begin{defn}[Concenterable upper periodic real sets]
An upper 1-periodic (real) set $A$ is called ``concenterable'' if there exists
a (constant) real number $\delta$ with the property that for every $a\in A$ there is a $k\in \mathbb{Z}$ such that
  $a-k\in [\delta,\delta+1)\cap A$ (equivalently, $(a+\mathbb{Z})\cap ([\delta,\delta+1)\cap A)\neq \emptyset$). Here, $\delta$ is called a
``concentration number of $A$'',
and we denote the set of all such $\delta$ by $\Coc(A)=\Coc_1(A)$. If $\delta\in \Coc(A)$, then
$[\delta,\delta+1)\cap A$ is called the $\delta$-center of $A$.
\end{defn}
It is obvious that $\delta\in \Coc(A)$ if and only if $a-\lfloor a-\delta\rfloor\in A$, for all $a\in A$.
Every 1-periodic set $A$ is concenterable and  $\Coc(A)=\mathbb{R}$. 
Now the most natural question that arises is what is the necessary and/or sufficient condition for
concenterability.
 Also, the most important challenge is to calculate and specify $\Coc(A)$.
 \begin{theorem}
 For every upper 1-periodic real set $A\neq\emptyset$, the followings are equivalent:\\
 $($a$)$ $A$ is concenterable;\\
 $($b$)$ There exists $\delta\in \mathbb{R}$ such that $a-\lfloor a-\delta\rfloor\in A$, for all $a\in A
  \cap [\delta+1,+\infty)$;\\
  $($c$)$ There exists a bounded above subset $\Lambda\subseteq A$ such that
 $(a+\mathbb{Z})\cap \Lambda\neq \emptyset$ for all $a\in A$;\\
 $($d$)$ There exists a bounded above and anti integer-transference subset $\Lambda'\subseteq A$ such that
 $(a+\mathbb{Z})\cap \Lambda'\neq \emptyset$ for all $a\in A$;\\
 $($e$)$ There exists a minimal bounded above subset $\Lambda''\subseteq A$ such that
 $(a+\mathbb{Z})\cap \Lambda''\neq \emptyset$ for all $a\in A$;\\
 $($f$)$ $\St(A)$ is bounded above;\\
  $($g$)$ $\Pf(A)$ is concenterable.
 \end{theorem}
 \begin{proof}
If $A=A+1$, then it is obvious. Thus, suppose that $A+1\subset A$.
 Due to $\mathbb{Z}_++A\subseteq A$, (a) and (b) are equivalent,
 and (b) implies (c), by putting $\Lambda:=[\delta,\delta+1)\cap A$.\\
Since the direct proof of this is a long case, we use the fundamental theorem.
Let $E,D$ be as mentioned in the representation $(2.3)$. Then, $D\dot{\cup} E$ satisfies the following conditions:\\
(1) $(a+\mathbb{Z})\cap (D\dot{\cup} E)\neq \emptyset$ for all $a\in A$;\\
(2) $D\dot{\cup} E$ is an anti integer-transference subset of $A$;\\
(3) $D\dot{\cup} E$ is minimal relative to property (1). For if $S\subset D\dot{\cup} E$
satisfies (1), then choosing $x\in (D\dot{\cup} E)\setminus S$ we have $x\in A$ and so $x-k\in S$,
for some integer $k$. Thus, $\{x,x-k\}\subseteq D$ or $\{x,x-k\}\subseteq E$, and hence
$k=0$ which implies $x\in S$ that is a contradiction).\\
On the other hand, if part (c) holds, then $E=A\setminus(A+1)$ is bounded above and vice versa.
For if, $\Lambda$ satisfies (c), then putting $y \sup\Lambda$ we  claim that
$\alpha$ is an upper bound for $E$. To obtain a contradiction, let
$x>y$ for some $x\in E$. Then we have $x\in A$, $x\not A+1$, and $x>y$.
Also, $x-k\in \Lambda$, for some $k\in \mathbb{Z}$ which implies $x-k\in A$ and $k\geq 1$
(because $x-k\leq \sup\Lambda =y<x$). So $x-1\in A$ that contradicts $x\not A+1$.
Conversely, if $E$ is bounded above, then we claim that $y:=\sup E$ is a concentration
number of $A$ and so $\Lambda=D\cup E$ satisfies (c). Let $x\in A$ and $x\geq y+1$, then $x\notin E$ and so
$x\in A+1$ that is $x-1\in A$. If $y\leq x-1<y+1$ we have $x-\lfloor x-y\rfloor\in A$, since
$\lfloor x-y\rfloor=1$, otherwise by the same reasoning for $x-1$ (instead of $x$) we conclude that
$x-2\in A$. We can continue this method until we reach the first $k\geq 2$ such that
$x-k\in A$ and $y\leq x-k<y+1$. this means $k=\lfloor x-y\rfloor$ and $x-\lfloor x-y\rfloor\in A$. Therefore,
the claim is proved, $\Lambda=D\cup E$ is bounded above and satisfies the conditions.\\
The above arguments clearly prove the equivalence of parts (a) till (f). Hence,
it is enough to show that (g) and (a) are equivalent or $\Coc(A)=\Coc(F_1(A))$.
If $\delta\in \Coc(F_1(A))\neq\emptyset$ and $x\in A$, then $x-\lfloor x-\delta\rfloor\in F_1(A)\subseteq A$
if $x\in F_1(A)$, otherwise $x-\lfloor x-\delta\rfloor\in \Pk_1(A)+\mathbb{Z}=\Pk_1(A)\subseteq A$.
Thus  $\delta\in \Coc(A)\neq\emptyset$.
Also, if $\delta\in \Coc(A)\neq\emptyset$ and $x\in F_1(A)$, then $x-\lfloor x-\delta\rfloor\in A=\Pk_1(A)\dot{\cup} F_1(A)$.
But $x-\lfloor x-\delta\rfloor\in \Pk_1(A)$ implies $x\in \Pk_1(A)\cap F_1(A)$ that is a contradiction. So
$x-\lfloor x-\delta\rfloor\in F_1(A)$. Thus $\delta\in \Coc(F_1(A))\neq\emptyset$.
 \end{proof}
 \begin{rem}
The most important result of the above theorem is the simple test of concenterability of $A\supseteq A+1$, i.e.,
$A\setminus(A+1)$ being bounded from above, and the following is a concentration number
\begin{equation}
 \delta_A =\left\{\begin{array}{l}
 \sup(A\setminus  (A+1))  \quad\; : \quad  \mbox{if } A\neq A+1 \\
0  \qquad \qquad \qquad \qquad : \quad \mbox{if } A=A+1
\end{array} \right.
\end{equation}
 which can be considered as the main concentration number. We may define $\delta(A)=+\infty$
 if and only if $A+1\subset A$ and $A\setminus(A+1)$ is not bounded above. Thus, $\delta(A)=+\infty$
 means $A$ is not 1-periodic and $A\setminus(A+1)$ is unbounded above.\\
As another important result, if $A\setminus(A+1)$ is bounded above, then there exists
a minimal subset of $A$ that is  bounded above, anti integer-transference subset $\Lambda'\subseteq A$ and
 $(a+\mathbb{Z})\cap \Lambda\neq \emptyset$ for all $a\in A$. It is interesting to know
 that every element of such subsets can be obtain by an integer shift of a unique element of $D\dot{\cup} E$.\\
It is worth noting that since for every arbitrary real set $A$, the summand set
$\Sigma_A:=\Sigma_{A|\mathbb{Z}_+}$ is upper 1-periodic and
$$
\Sigma_A\cap A=\Sigma_A+1\subseteq \Sigma_A\; , \; \Sigma_A\setminus A=\Sigma_A\setminus(\Sigma_A\cap A)=
\Sigma_A\setminus(\Sigma_A+1)=\St(\Sigma_A),
$$
we can put $\sigma_A:=\delta_{\Sigma_A}$ that is the main concentration number of $\Sigma_A$.
Therefore, if $A$ is an \underline{arbitrary} real set, then
\begin{equation}
 \sigma_A =\left\{\begin{array}{l}
 \sup(\Sigma_A\setminus A)  \qquad\;\;\;\; : \quad  \mbox{if } \Sigma_A\setminus A\neq \emptyset \\
0  \qquad \qquad \qquad \qquad : \quad \mbox{if } \Sigma_A\setminus A= \emptyset
\end{array} \right.
\end{equation}
One can see some of its applications in the topic of  ``limit summability of real functions''.
\end{rem}
\begin{theorem}[Characterization of concentration numbers of upper 1-periodic real sets]
 For every upper 1-periodic set $A$ we have
 \begin{equation}
\Coc(A) =\left\{\begin{array}{l}
 \mathbb{R}  \qquad\qquad\quad \;\;\; : \quad \;\; \mbox{if } A\setminus( A+1)=\emptyset \\
\mbox{$[$}\delta_A-1,+\infty)   \quad  : \quad\; \mbox{if } \delta_A\notin A\setminus( A+1)\neq\emptyset\\
(\delta_A-1,+\infty)   \quad  : \quad\; \mbox{if } \delta_A\in A\setminus( A+1)\\
\end{array} \right.
\end{equation}
with the convention that $[+\infty,+\infty)=(+\infty,+\infty)=\emptyset$.
 \end{theorem}
 \begin{proof}
Due to Theorem 2.11, 4.4 we have  $\Coc(A)=\Coc(F)$ where $F=\mathbb{Z}_+^0\dot{+}(A\setminus (A+1)$.
Hence, if $A=A+1$, then $\Coc(A)=\Coc(\emptyset)=\mathbb{R}$.
Also, if  $A\neq A+1$ and $A\setminus( A+1)$ is bounded above, then $\delta_A=\sup (A\setminus( A+1))<+\infty$.
Let $\delta\in \Coc(F)$ and put $E:=A\setminus (A+1)$, then for every $e\in E$ and $n\in \mathbb{Z}_+^0$ there exist
$e'\in E$ and $n'\in \mathbb{Z}_+^0$ such that $e+n-\lfloor e+n-\delta\rfloor=e'+n'$ and so
$e-\lfloor e-\delta\rfloor=e'+n'$ which implies $e-e'\in \mathbb{Z}$ and the property
$E\cap (E+\mathbb{Z}^*)=\emptyset$ requires $e=e'$. Thus $\lfloor e-\delta\rfloor=-n\leq 0$
and hence $\delta+1>e$ for all $e\in E$. Therefore, $\delta+1\geq\delta_A$ and so
$\delta\in(\delta_A-1,+\infty)$ if $\delta_A\in E$, and $\delta\in[\delta_A-1,+\infty)$ if $\delta_A\notin E$.
Conversely, if $\delta\geq \delta_A-1$ and $\delta_A\notin E$  (resp. $\delta> \delta_A-1$ and $\delta_A\in E$), then
 $\lfloor e-\delta\rfloor\leq 0$ for all $e\in E$ and so $e-\lfloor e-\delta\rfloor\in \mathbb{Z}_+^0+E=F$,
 which means $\delta\in \Coc(F)$.\\
Since for the case $A\neq A+1$ and $A\setminus( A+1)$ is unbounded above we have  $\Coc(F)=\emptyset$,
the proof is complete (considering the convention).
 \end{proof}
 \begin{cor}
If $A$ is an upper 1-periodic real set such that $\St(A)$ is bounded above and $\delta_A\notin \St(A) \neq\emptyset$,
then the least concentration number exists and equals to $\delta_A-1$.
\begin{ex}
Let $M$ be a fixed real number. Then
$$
\{x\in \mathbb{R}: y-\lfloor y-x\rfloor>M, \mbox{ for all $y>M$ } \}=(M,\infty),
$$
$$
\{x\in \mathbb{R}: y-\lfloor y-x\rfloor>M, \mbox{ for all $y\geq M$ }\}=[M,\infty).
$$
For proving these interesting identities, put $A=(M,+\infty)$ (resp. $A=[M,+\infty)$).
Since $A$ is upper 1-periodic and $\delta_A=M+1\in A\setminus( A+1)$ (rep.
$\delta_A=M+1\notin A\setminus( A+1)$), Theorem 4.6 completes the solution. 
It may be interesting for some readers to check the above identities directly.
\end{ex}
 \end{cor}
\subsection{Classification and characterization of all additive real semigroups and groups.}
As an application of the theory, we here give a useful classification of sub-semigroups (and subgroups)
of the additive group of real numbers. The classification gives us more information for such sub-structures.
As we mentioned before, for study of $b$-periodic types real sets ($b\neq 0$), it is enough to consider
the case $b=1$, and for study of  $1$-periodic types real sets it is enough to focus on upper $1$-periodics
(because $A$ is lower 1-periodic if and only if $-A$ is upper 1-periodic).
 \\
On the other hand, for study of sub-semigroups (resp. subgroup)
of real numbers it is enough to study all real upper 1-periodic (resp. 1-periodic) semigroups (resp. groups) $H$.  Because,
$H\dot{\leq}(\mathbb{R},+)$ (resp. $H\leq (\mathbb{R},+)$) if and only if $xH\dot{\leq} (\mathbb{R},+)$
(resp. $xH\leq (\mathbb{R},+)$), for every $x\in\mathbb{R}\setminus \{ 0\}$). Also,
we can assume that $1\in H$ without loss of generality, since
by choosing $x=\frac{1}{h}$ and putting $H_h:=\frac{1}{h}H$, for every $h\in H\setminus \{0\}$,
we have  $H=hH_h$ and $1\in H_h\dot{\leq}\mathbb{R}$ (resp. $1\in H_h\leq \mathbb{R}$)
and so $H_h$ is upper 1-periodic (resp. 1-periodic).\\
The wide variety of subgroups and subgroups of real numbers has made them seem out of reach.
However, as an application of this topic, with the unique direct representation of each real semigroup and group
(with their particular properties), they are classified as three disjoint classes.
But since such periodic subsets are also a sub-semigroup (algebraic sub-structure), we can find
a necessary and sufficient condition on $D$ and $E$ that make them characterized.
\begin{rem} Let $1+A\subseteq A$, then $A=(\mathbb{Z}\dot{+}D)\dot{\cup}(\mathbb{Z}_-^0\dot{+}E)$.\\
(a) Since directness of the summation $\mathbb{Z}_+^0\dot{+}E$ ($=\Pf(A)$) is equivalent to directness of $\mathbb{Z}\dot{+}E$,
$E$ can be assumed as a subset of a (fixed) transversal $\mathcal{E}$ of $\mathbb{Z}$ in $\mathbb{R}$ (we call it $E$-extended transversal of $\mathbb{Z}$).
Hence, $\mathcal{E}=\mathcal{E}(A)$ is \underline{dependent on $E=\St(A)$}, $\mathbb{R}=\mathbb{Z}\dot{+}\mathcal{E}$, $E\subseteq \mathcal{E}$, and every real number $x$ can be uniquely written as $x=P_\mathbb{Z}(x)+P_\mathcal{E}(x)$. Therefore, we have $P_{E}:\Pf(A)\longrightarrow E$,
$P_E=P_{\mathcal{E}}|_{\Pf(A)}$. Note that there is no such factor $\mathcal{E}$ that works for all upper 1-periodic sets (if such
$\mathcal{E}$ exists, then it must contains all $[M,M+1)=\St([M,+\infty))$, for every $M\geq 0$, that is impossible).\\
(b) We have two types of projections as:
$P_{D}:\Pk(A)\longrightarrow D$ (here $P_D=P_{[0,1)}|_{\Pk(A)}$), and $P_{[0,1)}$ is independent on $A$, $\Pk(A)$, and so it works
for all upper 1-periodic sets $A$.
\end{rem}
Now, we are ready to introduce a necessary and sufficient condition on $D,E$ in the representation for $A$ to
be a sub-semigroup.
\begin{theorem}
Suppose that $A\neq \emptyset$ is an upper 1-periodic subset of the additive group of real numbers, hence
it has the unique direct representation $A=(\mathbb{Z}\dot{+}D)\dot{\cup}(\mathbb{Z}_+^0\dot{+}E)$ where
$D\subseteq [0,1)$, and fix an $E$-extended transversal $\mathcal{E}$.
Then, $A$ is a sub-semigroup
if and only if for every $d,d'\in D$ and $e,e'\in E$:\\
$($a$)$ $\{d+d'\}, \{d+e\}\in  D$;\\
$($b$)$  $\{e+e'\}\in  D$ or $P_\mathcal{E}(e+e')\in E$.\\
$($Thus, if $A$ is a semigroup, then $D$ is either empty or additive, i.e., $\{d+d'\}\in  D$ for every $d,d'\in D$, $E$ is $P_\mathcal{E}$-additive
, i.e., $P_\mathcal{E}(e+e')\in E$ for every $e,e'\in E$, and $P_{[0,1)}(D+E)\subseteq D$. and vice versa.$)$
\end{theorem}
\begin{proof}
If (a) and (b) hold, then it is obvious that $(\mathbb{Z}+D)\cup(\mathbb{Z}_+^0+E)$ is a sub-semigroup.
Conversely, if $A$ is a sub-semigroup, then $\Pk_1(A)=\mathbb{Z}+D$ is either empty or an ideal of $A$.
The case  $\Pk_1(A)=\emptyset$ (equivalently $D=\emptyset$) implies the conditions (a) and (b) clearly (indeed, $e+e'\in \mathbb{Z}_++E$ for all $e,e'\in E$).
But if $\mathbb{Z}+D$ is an ideal of $A$, then
$d+d', d+e\in  D+\mathbb{Z}$,  $e+e'\in(\mathbb{Z}+D)\cup(\mathbb{Z}_+^0+E)$ and so we arrive at (a), (b).
\end{proof}
\begin{defn}
We call a couple $(D,E)$ of subsets of real numbers an ``additive couple'' if\\
(a) $D\subseteq [0,1)$, $D\cap E=\emptyset, D\cup E\ne \emptyset$;\\
(b) $D\cup E$ is anti integer-transference;\\
(c) $D$ is additive, $E$ is $P_\mathcal{E}$-additive respect to an $E$-extended transversal $\mathcal{E}$,
and $P_{[0,1)}(D+E)\subseteq D$.
\end{defn}
\begin{cor} Let $D,E$ be subsets of real numbers.\\
(a) $(D,E)$ is an additive couple if and only if
$$\mathbb{Z}+D\cup\mathbb{Z}_+^0+E=\mathbb{Z}\dot{+}D\dot{\cup}\mathbb{Z}_+^0\dot{+}E\dot{\leq} (\mathbb{R},+).$$
Under the assumption of Theorem 4.11:\\
(b) $A$ is a sub-semigroup if and only if $(D,E)$ is an additive couple.\\
(c) $A$ is a subgroup of $(\mathbb{R},+)$ if and only if
$E=\emptyset$ and one of the following equivalent conditions hold:\\
$($i$)$ $D$ is additive and $1-d\in D$ for all $d\in D\setminus \{0\}$ $($i.e, $D$ is 1-symmetric$)$;\\
$($ii$)$ $P_{[0,1)}(D+D)\subseteq D$ and $1-(D\setminus \{0\})\subseteq D$;\\
$($iii$)$ $D$ is subtractive $($i.e., for every $d,d'\in D$, $\{d-d'\}\in  D)$;\\
$($iv$)$ $P_{[0,1)}(D-D)\subseteq D$.\\
$($See \cite{grplike} and its references for more information$)$
\end{cor}
\begin{cor}
The class of all upper 1-periodic semigroups is
$$\{(\mathbb{Z}\dot{+}D)\dot{\cup}(\mathbb{Z}_+^0\dot{+}E): \mbox{$(D,E)$ is an additive couple}\},$$
Hence, the class of all semigroups containing 1 is
$$\{(\mathbb{Z}\dot{+}D)\dot{\cup}(\mathbb{Z}_+^0\dot{+}E): \mbox{$(D,E)$ is an additive couple and $0\in D\cup E$}\}.$$
\end{cor}
Now, we are ready to state the classification of sub-semigroups (and subgroups) of $(\mathbb{R},+)$ via the classification
of upper periodic sets.\\
\textbf{First class real semigroups ($1$-periodic case)}. A sub-semigroup $H$ lies in this class if and only if $\Pk_\mathbb{Z}(H)=H$.
 Then,
$H=\mathbb{Z}\dot{+}D$ where $D$ is a unique decimal set. It is important to now that this class contains all sub-semigroups (resp. subgroups)
containing 0,1 (resp. 1). Hence, due to Theorem 4.11 and Corollary 4.13, we arrive at the following theorem.
\begin{theorem}
The first class real semigroups is equal to
$$\{\mathbb{Z}\dot{+}D: \mbox{$D$ is an additive subset of $[0,1)$}\}.$$
Moreover, $1\in H=\mathbb{Z}\dot{+}D$ if and only if $0\in D$, and it contains the class of all real groups containing 1, that is
$$\{\mathbb{Z}\dot{+}D: \mbox{$D$ is a subtractive subset of $[0,1)$}\}.$$
\end{theorem}
\begin{ex}
Fix a positive integer $n$ and put $\mathbb{D}_n:=\{0,
\frac{1}{n},\cdots,\frac{n-1}{n}\}$, $H:=\langle\langle \mathbb{D}_n\rangle\rangle$ (i.e., the subgroup generated by $\mathbb{D}_n$).
Then, $H$ is a real group containing 1 and so a first class semigroup. The (unique) decimal set $D$ in the unique direct representation
$H=\mathbb{Z}\dot{+}D$ is the same $\mathbb{D}_n$.
\end{ex}
This example also raises an interesting issue as follows:\\
\textbf{Problem VII.} Characterize all $\Omega \subseteq \mathbb{R}$ such that  $\langle\langle \Omega\rangle\rangle=\mathbb{Z}\dot{+}\Omega$.\\
\textbf{Second class real semigroups ($1$-periodic free case).} This class consists of all upper
$1$-periodic semigroups $H$ such that $\Pk(H)=\emptyset$  (i.e., $H$ is $\mathbb{Z}$-periodic free).
 They have the unique direct representation
$H=\mathbb{Z}_+^0\dot{+} E$, and we arrive at the following theorem.
\begin{theorem}
Let $H\neq \emptyset$ be a 1-periodic free and an upper 1-periodic subset of real numbers, hence it has the direct representation $H=\mathbb{Z}_+^0\dot{+}E$,
for a unique $E\subseteq \mathbb{R}$. Fix an $E$-extended transversal $\mathcal{E}$ of $\mathbb{Z}$. Then,
$H$ is a sub-semigroup of $(\mathbb{R},+)$ if and only if $E=\St(H)$ is $P_\mathcal{E}$-additive.
Therefore, the second class real semigroups is equal to
$$\{\mathbb{Z}_+^0\dot{+}E: \mbox{$E$ is $P_\mathcal{E}$-additive, for some $($all$)$ $E$-extended transversal $\mathcal{E}$}\},$$
and none of the members of it is a subgroup.
\end{theorem}
Note that if $H$ in the above is a subgroup, then $1\in H$ and so $Z\subseteq \Pk(H)\neq\emptyset$ that is a contradiction.
\begin{ex}
Fix an irrational number $\alpha$ and put $H:=\mathbb{Z}_+ + \mathbb{Z}\alpha$. Then
$$
\Pk(H)=\bigcap_{k\in \mathbb{Z}}H+k=\bigcap_{k\geq 1}H+k=\bigcap_{k\geq 1}(\mathbb{Z}_++k) + \mathbb{Z}\alpha=\emptyset.
$$
Therefore $H$ is a second class real semigroup and $E=\St(H)=1+\mathbb{Z}\alpha=\{0,1\pm \alpha, 1\pm 2\alpha, \cdots \}$. Thus it has the unique direct representation
$$
H=\mathbb{Z}_+^0\dot{+}(1+\mathbb{Z}\alpha)
$$
Note that here, $E$ is $\mathcal{E}$-additive related to every transversal $\mathcal{E}$ of $\mathbb{Z}$
containing $\{0,1\pm \alpha, 1\pm 2\alpha, \cdots \}$.
\end{ex}
\textbf{Third class real semigroups (1-mixed case).} The $\mathbb{Z}$-mixed case ($\emptyset\neq \Pk(H) \neq H$).
Similar to the two pervious classes of real semigroups  we have the following unique direct representation.
\begin{theorem}
Let $H\neq \emptyset$ be an upper 1-periodic subset of real numbers that is  non 1-periodic and non 1-periodic free. Hence it has the direct representation $H=(\mathbb{Z}\dot{+}D)\dot{\cup}(\mathbb{Z}_+^0\dot{+}E)$,
for a unique $\emptyset\neq E\subseteq \mathbb{R}$ and a unique $\emptyset\neq D\subseteq [0,1)$.
Fix an $E$-extended transversal $\mathcal{E}$ of $\mathbb{Z}$. Then,
$H$ is a sub-semigroup of $(\mathbb{R},+)$ if and only if $(D,E)$ is an additive couple.
Therefore, the third class real semigroups is equal to
$$\{(\mathbb{Z}\dot{+}D)\dot{\cup}(\mathbb{Z}_+^0\dot{+}E): \mbox{$(D,E)\neq (\emptyset,\emptyset)$ is an additive couple}\},$$
and non of the members of it is a subgroup.
\end{theorem}
Note that if $H$ in the above is a subgroup, then $H$ has two representations $H=(\mathbb{Z}\dot{+}D)\dot{\cup}(\mathbb{Z}_+^0\dot{+}E)$,
 $H=\mathbb{Z}\dot{+}D'$, and so $E=\emptyset$ that is a contradiction.
\begin{ex}
For $\mathbb{Q}$-linearly independent numbers $1,\alpha,\beta$ put
$$
H:=(\mathbb{Z} + \mathbb{Z}\alpha+\mathbb{Z}_+\beta)\cup (\mathbb{Z}_+ + \mathbb{Z}\alpha)
$$
Then $\Pk(H)=\mathbb{Z} + \mathbb{Z}\alpha+\mathbb{Z}_+\beta$ and $\Pf(H)=\mathbb{Z}_+ + \mathbb{Z}\alpha$ (see Subsection 2.1
and Lemma 2.16) and so $H$ is a 1-mixed real semigroup. Thus $(\mathbb{Z} + \mathbb{Z}\alpha+\mathbb{Z}_+\beta,\mathbb{Z}_+ + \mathbb{Z}\alpha)$
is an additive couple (related to $\mathcal{E}$ in Example 4.16).
\end{ex}
We expect to obtain similar properties for the multiplicative real semigroups as follows.\\
{\bf Project VIII.}
Do similar studies for periodic subsets, subgroups and sub-semigroups of the multiplicative group $(\mathbb{R}^*,\cdot)$
by using the isomorphism $\frac{\mathbb{R}^*}{\{-1,1\}}\cong (\mathbb{R},+)$ or directly.\\
{\bf More properties of real semigroups and groups.}
The following lemma and remark improve Lemma 3.2 of \cite{ulpss}. It shows some important facts about sub-semigroups of real and rational numbers.
\begin{theorem} Let $H$ be an upper $1$-periodic (additive) real semigroup of the third class
$($i.e., $\emptyset\subset\Pk(H)\subset H)$. Then\\
$($a$)$ The unique direct representation $(2.5)$ with the additive couple $(D,E)$ holds, and $H$ is not a subgroup.\\
$($b$)$ $Pk(H)$ is an irrational sub-semigroup of $(\mathbb{R},+)$ and also an irrational ideal of $H$.\\
$($c$)$ $(\Pf(H)\cap \mathbb{Q})+\Pf(H)\subseteq \Pf(H)$ $($i.e., $\Pf(H)$ is upper $\Pf(H)\cap \mathbb{Q}$-periodic$)$. Hence,
$\Pf(H)\cap \mathbb{Q}$ is a rational semigroup if it is nonempty $($e.g. $1\in H)$. Also if $\Pf(H)\subseteq \mathbb{Q}$,
then $\Pf(H)$ is a $($rational$)$ semigroup.
 \end{theorem}
\begin{proof} The first part is obtained from Theorem 2.11, 4.11.\\
Now, put $C=\Pk(H)$ and $F=\Pf(H)$. If $C$ contains a rational member $r$, then
 $\{r\}\in \mathbb{Z}+r\subseteq C\subseteq H$ and so $\mathbb{Z}\subseteq \mathbb{Z}_+\{r\}+\mathbb{Z}\subseteq H$
 which requires $1+H=H$ that is a contradiction. Since $C$ is an ideal of $H$, (b) is proved.\\
The property $F\cap \mathbb{Q}=H\cap \mathbb{Q}$ is
clear, considering $C\subseteq \mathbb{Q} ^c$, $C\cap F=\emptyset$.\\
 If $x\in F$,  $r\in F\cap \mathbb{Q}$, and $x+r\in C$, then $H+C\subseteq C$,
$\mathbb{Z}+C=C$ imply $x+nr+k\in C$, for all positive integers
$n$ and integers $k$. Thus $x\in C$ and this is a contradiction.
Therefore, the first part of (c) is proved.  Now, one can cpnclude (c) easily.
\end{proof}
The following is an important property for semigroups of rationales.
\begin{cor}
The additive group of rationales does not contain any third class sub-semigroup. Hence,
 if $1\in H\dot{\leq} \mathbb{Q}$, then either $\Pk(H)=H$ or $\Pk(H)=\emptyset$.
\end{cor}
\begin{cor}
Let $H$ be a real semigroup containing 1, and $C$, $F$ be
nonempty subsets of $H$ satisfying the following conditions:\\
a) $H=C\cup F$; \\
b) If $x\in C$, then $x+k\in H$ for all integers $k$;\\
c) For every $x\in F$ there exists an integer $k_0$ such that
$x+k_0\notin H$. \\
Then $C+H\subseteq C\subseteq \mathbb{Q} ^c$ (i.e. $C$ is an
irrational ideal of $H$) and $F\cap \mathbb{Q}=H\cap \mathbb{Q}$.
Moreover, if $F\subseteq \mathbb{Q}$, then $F$ is a rational
semigroup.
\end{cor}
\begin{rem} If $H$ is an upper 1-periodic semigroup $($e.g., $1\in H\dot{\leq}\mathbb{R})$, then \\
(i) $\Pk(H)$ is either empty or an irrational ideal of $H$ (by the above theorem),\\
(ii) $\Pf(H)\cap \mathbb{Q}$ is either empty or a rational sub-semigroup (by the above theorem),\\
(iii) but $\Pf(H)$ does not need to be a sub-semigroup.\\
Putting
$$ H=\{ m\sqrt{2}+k|m,k\in \mathbb{Z}, m\geq 2\}\cup (\mathbb{Z}_++\sqrt{2}),$$ we have
 $Pk(H)=\{m\sqrt{2}+k|m,k\in \mathbb{Z}, m\geq 2\}$,
$Pf(H)=\mathbb{Z}_++\sqrt{2}$ and $H$ is a $1$-mix real semigroup,
but $Pf(H)$ is not a real semigroup.\\
(iv) $H$ may be concenterable or non-concenterable, hence a study of concenterable real semigroups should be
useful and interesting. Also $\St(H)$ may be bounded or unbounded below. For example, the semigroup in
part (iii) is concenterable and $\St(H)$  bounded below, but
 $$H:=(\mathbb{Z} + \mathbb{Z}\sqrt{2} +\mathbb{Z}_+\sqrt{3})\cup (\mathbb{Z}_+ + \mathbb{Z}\sqrt{2}) $$
 is non-concenterable and $\St(H)$  unbounded below.
\end{rem}
{\bf Construction of mixed real semigroups by periodic and periodic free semigroups.}
Let $H_1,H_2$ be semigroups of the first and second classes, respectively, and put
$H:=H_1+H_2\cup H_2$. Then, $H$ is an upper 1-periodic semigroup and $\Pk(H)=H_1+H_2$. Also,
we have
$$
\Pf(H)=H_2\Leftrightarrow H_1+H_2\cap H_2=\emptyset \Leftrightarrow H_1\cap(H_2-H_2)=\emptyset.
$$
Therefore, if $H_1,H_2$ be semigroups of the first and second classes such that $H_1\cap(H_2-H_2)=\emptyset$,
then $H=H_1+H_2\cup H_2$ is a third class semigroup. Moreover, $\St(H)=\St(H_2)$ and hence $H$ is concenterable
if and only if $H_2$ is so. Hence, this is an appropriate method for constructing a vast subclass of the third class of semigroups.
According Example 4.20, if  $H_1:=\mathbb{Z} +\mathbb{Z}_+\beta$
and $H_2:=\mathbb{Z}_+ + \mathbb{Z}\alpha$, then $H=H_1+H_2\cup H_2$ is the same as $H$ in that example
which provides infinitely many cases of such semigroups.
An study of such semigroups should be useful in the topic of real semigroups.\\
We end this section by another suggested research project for real semigroups. \\
{\bf Project IX.} We call a subset $A$ of $(\mathbb{R},+)$ totally concenterable if it is upper $b$-periodic
for some $b\neq 0$ and $\frac{1}{b}A$ is concenterable for all  $b\neq 0$ such that $b+A\subseteq A$.
For example, all subsets of the form $(M,+\infty)$ and $[M,+\infty)$ are totaly concenterable (why?, also
give some examples of real sets $A$ such that they are not totally concenterable although they are concenterable).
A study of totally concenterable subsets, and specially
concenterable (and totally concenterable) real semigroups is of special interest (for obtaining some classes of real
semigroups).
\section{Other expected future studies and projects}
If the left factor subgroup $\mathcal{B}$ is normal, then we have more properties
as follows.
\begin{theorem}[Cancellative semigroups containing a normal factor subgroup]
Every left cancelative semigroup $M$ containing a left normal factor subgroup ${\mathcal{B}}$ is a monoid.
Also, we have
$1_M=1_{\mathcal{B}}$,
 $M=\mathcal{B}\cdot\mathcal{D}=\mathcal{D}\cdot\mathcal{B}$ and
$\mathcal{B}\cap\mathcal{D}$ is a singleton.
\end{theorem}
\begin{proof}
It is worth noting that $\mathcal{B}\trianglelefteq M$ implies $\mathcal{B}A=A\mathcal{B}$, for
all $A\subseteq M$. Thus $M=\mathcal{B}\cdot \mathcal{D}=\mathcal{D}\mathcal{B}$
and we claim that the product $\mathcal{D}\mathcal{B}$ is also direct. For
if  $d_1b_1=d_2b_2$ where $d_1,d_2\in \mathcal{D}$, $b_1,b_2\in \mathcal{B}$, then
there are $b'_1,b'_2\in \mathcal{B}$ such that $b'_1d_1=b'_2d_2$ and so
$b'_1=b'_2$ (since the product $\mathcal{B}\mathcal{D}$ is  direct), and then
$d_1=d_2$ (because $M$ is left cancelative).
Now, if $x\in M$, then
$x=\beta d$, for some $\beta\in \mathcal{B}$ , $d\in \mathcal{D}$, and so
$$
1_{\mathcal{B}}x=1_{\mathcal{B}}(\beta d)=(1_{\mathcal{B}}\beta)d=\beta d=x,
$$
similarly we have $x=x1_{\mathcal{B}}$.
Therefore $M$ is a monoid with the identity $1_{\mathcal{B}}$. Also, we claim that $\mathcal{B}\cap \mathcal{D}$
is a singleton. Because $1_\mathcal{B}=\beta_0 d_0$, for some $\beta_0\in \mathcal{B}$
and $d_0\in D$, so $d_0=\beta_0^{-1}\in \mathcal{B}\cap\mathcal{D}$. Now, if
$\gamma \in \mathcal{B}\cap\mathcal{D}$, then $\gamma =1_\mathcal{B}\gamma =\beta_0 (d_0\gamma)=(\beta b)d_0$,
for some $b\in \mathcal{B}$,
thus $\gamma=d_0$. Therefore $\mathcal{B}\cap\mathcal{D}$ is a singletons since $d_0$ is unique.
\end{proof}
The above lemma is the starting point of our future general study with many research projects which some of them are as follows.\\
{\bf Project X (two-sided unique direct representation of two-sided upper periodic sets).}  This should be
done similar to the mentioned topics and theorems for the left upper periodic subsets, of course, by
using the above theorem. We recall that $A$ is (two-sided) upper $B$-periodic if and only if it is both left and right
upper $B$-periodic. Hence, $A$ is upper $B$-periodic if and only if $BA\cup AB\subseteq A$. This also induces another extensive
research project as follows.  \\
{\bf Project XI (study of middle $B$-periodic and bi-$B$-periodic type subsets).} Fix a subset $B$. We call a subset $A\subseteq S$
middle $B$-periodic (resp. bi-$B$-periodic) if $ABA=A$ (resp. $BAB=A$). Also it is called
upper middle $B$-periodic (resp. upper bi-$B$-periodic) if $ABA\subseteq A$ (resp. $BAB\subseteq A$), and the lower case is defined analogously.
It is clear that every $B$-periodic (resp. upper $B$-periodic) set is bi-$B$-periodic (resp. upper bi-$B$-periodic). Also, every left or right
 upper $B$-periodic sub-semigroup is upper middle $B$-periodic.
 These definitions and relations
 arrive us in an interesting  extensive research project that requires a lot of effort. We propose introducing middle periodic an bi-periodic
 kernels of subsets, related summand sets (and so on), for the study at first (similar to the left and right cases).\\
{\bf Project XII (classification of sub-semigroups and subgroups containing a fixed element or subset).}  The basic idea is that
$BY\cup YB\subseteq Y$ (i.e. $Y$ is two-sided upper $B$-periodic),  for every sub-semigroup $Y$ containing $B$. Hence,
classification of sub-semigroups and subgroups containing a fixed subset that its special case has been studied in the previous section
 for real semigroups and groups (also see Subsection 2.1, 2.2) is expected (followed by the previous project).
 \\
In the following projects by $\mathcal{L}=\mathcal{L}(X)$ we mean the set of all left identities of $X$ (that may be empty).\\
{\bf Project XIII (totaly non periodic and non upper periodic sets).} First, note that every subset $A$ is left $B$-periodic
for all $B\subseteq \mathcal{L}(X)$. Now, we call $A\subseteq X$ totaly non left periodic (resp. totaly non left upper periodic)
if for every $B\subseteq X$, $BA=A$ (resp. $BA\subseteq A$) implies $B\subseteq \mathcal{L}(X)$. For example, every bounded interval
of the real line is totaly non left upper periodic (and so totaly non left periodic). A study of such subsets should be the goal
of this research project.  \\
{\bf Project IXX (totaly periodic and upper periodic free sets).}
We call $A\subseteq X$ totaly left periodic free (resp. totaly left upper periodic free)
if all its non-empty subsets are totaly non left periodic (resp. totaly non left upper periodic).
Similar definitions can be mentioned for the right, two-sided, and middle cases.
Hence $A$ is totaly left periodic free (resp. totaly left upper periodic free) if and only
if $\Pk^\ell_B(A')=\emptyset$ (resp. $\Upk^\ell_B(A')=\emptyset$) for all $B\nsubseteq \mathcal{L}(X)$ and $A'\subseteq A$.
For example, every bounded interval
of the real line is totaly left upper periodic free (and so totaly left periodic free). This research project is closely related to
Project XII and X.  \\
{\bf Project XX (union and intersection of the periodic and upper periodic kernels).}  In study of totaly periodic and upper periodic
(resp. periodic free) subsets
and also sub-structures, a study of
$$\bigcup_{B\subseteq X}\Pk_B(A)\; , \; \bigcap_{B\subseteq X}\Pk_B(A)\; , \;\bigcup_{B\subseteq X}\Upk_B(A)\; , \; \bigcap_{B\subseteq X}\Upk_B(A)$$
should be useful. Note that all the above unions and intersections are subsets of $A$, and the three cases equal to the empty
set, $A$, or non-empty and proper subset of $A$ are important subjects for the research project.\\
{\bf Project XXI (characterization or classification of totaly periodic free sub-semigroups).}  This project can be done after Project XIII
and Project IXX.  We hope this also be useful for characterization of the real sub-semigroups.\\
{\bf Project XXII (applications to the functional equations on algebraic structures).} The projections induced by the direct representations
of the periodic types also satisfy some important functional equations. Motivated by this, Project X, and some results of \cite{new}, we hope to solve
some related functional equations (on some algebraic structures) such as $f(x\mu(y))=f(\mu(x)y)=f(xy)$ and
$\mu(\mu(xy)z)=f(\mu(xy)z)=f(xyz)=f(x\mu(yz))=\mu(x\mu(yz))$ where $f$ and $\mu$ are unknown functions including the case $f=\mu$.

\end{document}